\documentclass[11pt]{amsart}
\usepackage{amssymb,amsmath,amsfonts,amscd,euscript}
\usepackage[all]{xy}
\usepackage{color}

\usepackage[backref=page]{hyperref}
\usepackage{todonotes}

\usepackage{amsthm,enumitem,bbm,pict2e,xcolor}
\usepackage{indentfirst}

\newcommand{\nc}{\newcommand}

\numberwithin{equation}{section}
\newtheorem{thm}{Theorem}[section]
\newtheorem{prop}[thm]{Proposition}
\newtheorem{lem}[thm]{Lemma}
\newtheorem{cor}[thm]{Corollary}
\newtheorem{rem}[thm]{Remark}

\newtheorem{example}[thm]{Example}

\newtheorem{theoremA}{Theorem}

\nc{\gl}{\mathfrak{gl}}
\nc{\GL}{\mathfrak{GL}}
\nc{\g}{\mathfrak{g}}
\nc{\gh}{\widehat\g}
\nc{\h}{\mathfrak{h}}
\nc{\la}{\lambda}
\nc{\La}{\Lambda}
\nc{\al}{\alpha }
\nc{\be}{\beta }
\nc{\ve}{\varepsilon }
\nc{\om}{\omega }

\nc{\ta}{\theta}
\nc{\veps}{\varepsilon}
\nc{\ch}{{\mathop {\rm ch}}}
\nc{\Tr}{{\mathop {\rm Tr}\,}}
\nc{\Id}{{\mathop {\rm Id}}}
\nc{\bra}{\langle}
\nc{\ket}{\rangle}
\nc{\x}{{\bf x}}
\nc{\bs}{{\bf s}}
\nc{\bp}{{\bf p}}
\nc{\bt}{{\bf t}}
\nc{\pa}{\partial}
\nc{\ld}{\ldots}
\nc{\cd}{\cdots}
\nc{\hk}{\hookrightarrow}
\nc{\T}{\otimes}
\nc{\gr}{\mathrm{gr}}
\nc{\Gr}{\mathrm{Gr}}
\nc{\ov}{\overline}

\nc{\cO}{\mathcal O}
\nc{\cF}{\mathcal F}
\nc{\cL}{\mathcal L}
\nc{\msl}{\mathfrak{sl}}
\nc{\mgl}{\mathfrak{gl}}
\nc{\U}{\mathrm U}
\nc{\V}{\EuScript V}
\nc{\bH}{\EuScript H}
\nc{\Res}{\mathrm{Res\ }}

\newcommand{\bC}{{\mathbb C}}
\newcommand{\bR}{{\mathbb R}}
\newcommand{\bZ}{{\mathbb Z}}

\newcommand{\bP}{{\mathbb P}}

\newcommand{\bG}{{\mathbb G}}
\newcommand{\fp}{{\mathfrak p}}
\newcommand{\fg}{{\mathfrak g}}
\newcommand{\fb}{{\mathfrak b}}
\newcommand{\fh}{{\mathfrak h}}

\newcommand{\fn}{{\mathfrak n}}
\newcommand{\fa}{{\mathfrak a}}

\newcommand{\br}{{\bf r}}

\newcommand{\Fl}{\EuScript{F}}
\DeclareMathOperator{\Hom}{Hom}
\DeclareMathOperator{\spa}{span}

\newcommand{\bm}{{\bf m}}

\begin{document}

\title[Truncated Grassmannians]
{Truncated Grassmannians: polytopes, representations, and Cauchy formulas}

\author[Feigin]{Evgeny Feigin}
\address{Evgeny Feigin:\newline
School of Mathematical Sciences, Tel Aviv University, Tel Aviv
69978, Israel
}
\email{evgfeig@gmail.com}

\author[Makhlin]{Igor Makhlin}
\address{Igor Makhlin: 
% \newline
% Technische Universit\"at Berlin, Fakult\"at II Mathematik und Naturwissenschaften, Institut f\"ur Mathematik, FG Diskrete Mathematik/Geometrie
}
\email{iymakhlin@gmail.com}

\begin{abstract}
The goal of this paper is to develop algebraic and combinatorial techniques
for studying the geometric properties of truncated Grassmannians, a class of algebraic
varieties obtained via certain projections of classical Grassmannians within their
Pl\"ucker embeddings. We introduce a class of cyclic representations of maximal 
parabolic subalgebras of the general linear Lie algebra, which are proved to serve as 
duals of the homogeneous components of coordinate rings of truncated 
Grassmannians. We prove that the characters of these representations are computed via
truncated Cauchy formulas and establish a connection with coordinate rings of
complete collineations. We also construct monomial bases of these representations
in terms of certain generalized chain polytopes and use the RSK correspondence to 
establish connections with Cauchy-type formulas and with Gelfand--Tsetlin 
patterns. 
\end{abstract}

\maketitle

\section{Introduction}
Classical Grassmann varieties play an important role in various fields of
mathematics \cite{Fu97,H92,MS74} due to a vast number of connections and
applications. In particular, their geometric and topological
properties are closely related to the representation theory of $\mgl_n$ and 
can be described via deep and beautiful combinatorics. The goal of this paper is to
extend the algebraic and geometric constructions to the larger family of truncated 
Grassmannians.

Let $V$ be an $n$-dimensional complex vector space with a basis $\{v_i\}_{i=1}^n$. 
We fix $d<n$. To simplify notation, we always assume that $d\le n-d$. 
The classical Grassmann variety $\Gr_d(V)$ 
admits the Pl\"ucker embedding into the projective space $\bP(\bigwedge^dV)$.
The wedge power $\bigwedge^dV$ is isomorphic to the $d$-th fundamental representation
$V_{\om_d}$; in particular, the highest weight vector is equal to the wedge product
$v_1\wedge\dots\wedge v_d$ of the first $d$ basis vectors of $V$.

We define the degree of a wedge monomial in the $v_i$ to be
the number of indices larger than $d$. This degree coincides with the 
PBW degree which counts how far a vector is from the highest weight vector \cite{FFL11,Fe12,Fe23}.
For each $r=0,\dots,d$, we define the space $V^r = V_{\om_d}^r$ as the
quotient of $V_{\om_d}$ by the span of all wedge monomials of degree larger than $r$.
In particular, $V^0$ is one-dimensional and $V^d$ coincides with $V_{\om_d}$. 
One easily sees that all $V^r$ are acted upon by the maximal parabolic
subalgebra $\fp^-_d\subset\mgl(V)$ and that all $V^r$ are cyclic with respect to the
abelian radical $\fa_d\subset \fp^-_d$. The parabolic subalgebra $\fp_d^-$ contains the 
direct sum $\mgl(V_d)\oplus \mgl(V_{n-d})$, where $V=V_d\oplus V_{n-d}$ is the direct sum
decomposition into the span of the first $d$ and the last $n-d$ basis vectors.

The truncated Grassmannian $\Gr^r\subset \bP(V^r)$ 
is defined as the closure of the $\exp(\fa_d)$-orbit passing through the highest weight 
line $[v^r]$ \cite{FO24,Fe26}. Truncated Grassmannians are closely related to the partial versions of spaces 
of complete collineations \cite{L88,MMMSV24,Vain84,Th99} and are useful for the study of coupled cluster theory \cite{FSS25,FGNP26,Sv25} and 
blow-ups along Schubert varieties \cite{BS25,Fa26,FW25,Fe26}. 

Clearly, the varieties $\Gr^r$  are acted upon by the $d$-th maximal parabolic subgroup of $GL_n$.
We introduce the complete truncated Grassmannian
$\Gr^{[d]}$ as the closure of the $\exp(\fa_d)$-orbit of the product of highest weight lines inside the 
product of all truncated Grassmannians $\Gr^r$, $1\le r\le d$. Our first goal is
to study its coordinate ring. 

Let $\bm=(m_1,\dots,m_d)$ be a collection of nonnegative integers. 
Recall that a cyclic tensor product $M_1\odot M_2$ of two cyclic modules of a Lie algebra $\fa$ 
is the $\U(\fa)$-span of the tensor product of cyclic vectors inside $M_1\T M_2$.
Our central algebraic
object of study is the cyclic tensor product $V_\bm$:
\[
V_\bm = (\underbrace{V^1\odot\dots\odot  V^1}_{m_1}) \odot \dots \odot 
(\underbrace{V^d\odot\dots\odot V^d}_{m_d}).
\]
The direct sum $\bigoplus_{\bm} V_\bm^*$ admits a natural structure of an associative 
commutative algebra generated by the dual spaces $(V^r)^*$, $1\le r\le d$.
If all $m_i$ but $m_1$ and $m_d$ vanish, then 
we recover the spaces studied in \cite{BSS25,Fe24}.

In order to state our first theorem, recall the weak dominance order on partitions: 
we say that $\la\le \mu$ if for all partial sums $\la_1+\dots +\la_i\le \mu_1+\dots +\mu_i$
(we do not require that $\la$ and $\mu$ have the same size). We  define the partition
$\Lambda_\bm$ to have $m_r$ columns of height $r$ for each $r$. The notation $\Sigma_\la$
stands for the classical Schur functor; hence, $\Sigma_\la V$ is the irreducible $\mgl(V)$
module corresponding to the partition $\la$.

\begin{theoremA}
The multihomogeneous coordinate ring of $\Gr^{[d]}$ is isomorphic to the direct sum of all
$V_\bm^*$. 
The space $V_\bm$  admits the following decomposition as an $\msl(V_d)\oplus \mgl(V_{n-d})$-module:
\[
V_\bm \simeq \bigoplus_{\substack{\la\le\Lambda_\bm\\ \ell(\la)\le d}} 
\Sigma_\la V_d^*\T \Sigma_\la V_{n-d}.
\]
\end{theoremA}

One may turn this decomposition into an isomorphism of $\mgl(V_d)\oplus \mgl(V_{n-d})$-modules by tensoring the left-hand side with a one-dimensional $\mgl(V_d)$-module. 
Furthermore, Theorem A is actually obtained in the greater generality of \textit{partial truncated Grassmannians} $\Gr^\br$ with $\br\subset[d]$, defined analogously. 

In order to study the representations $V_\bm$, we introduce our main combinatorial objects of study: a family
of convex polytopes. Let $P$ be a finite poset of width $w(P)=d$. 
For $\bm\in \bZ_{\ge 0}^d$, we introduce the convex polytope $X_\bm\subset \bR_{\ge 0}^P$ defined by the following 
inequalities labeled by all subposets $P'\subset P$:
\[
\sum_{\al\in P'} x_\al \le \sum_{i=1}^d \min(i,w(P'))m_i.
\]

We apply the above construction to the poset $P=P_d$ consisting of pairs $(i,j)$ with
$i\in [d+1,n]$, $j\in [1,d]$ and with $(i,j)\le (i_1,j_1)$ if and only if 
$i\le i_1$, $j\le j_1$. The elements of this poset label the matrix units $E_{i,j}\in \fa_d$.
Hence, any point $\bs\in S_\bm = X_\bm\cap \bZ^{P_d}$ naturally defines a vector in $V_\bm$, which
is equal to $E^\bs v_\bm = \prod_{(i,j)\in P_d} E_{i,j}^{s_{i,j}} v_\bm$.

\begin{theoremA}
The monomials $E^\bs v_\bm$, $\bs\in S_\bm$ form a basis of $V_\bm$. 
\end{theoremA}

The proofs of the two theorems above use the RSK correspondence \cite{Kn70} and Greene's 
theorem \cite{G74}. This technique also allows for a significantly more conceptual proof of the Minkowski sum property result in 
\cite{Fe24} (see also \cite{Fe25}) for the case $m_2=\dots=m_{d-1}=0$. We note that in general this property does
not hold. The RSK algorithm also leads to the following description of the polytopes $X_\bm$.

\begin{theoremA}
The polytope $X_\bm$ admits a piecewise linear bijection with a certain subpolytope of the 
Gelfand--Tsetlin polytope of a multiple of the fundamental weight $\om_d$. This subpolytope is defined
by explicit inequalities only involving entries in a single row of the GT pattern.
\end{theoremA}

Our last result is a description of higher truncated fundamental modules in terms of
the spaces $V_\bm$ and a description of higher truncated Grassmannians. For an opposite 
Demazure module $D_{\sigma\la}\subset V_\la$ in an irreducible highest 
weight module $V_\la$, one defines the truncated space $V^\sigma_\la$ as 
$V_\la/D_{\sigma\la}$. The corresponding truncated flag variety $\Fl^\sigma_\la$ is the closure of the 
opposite Borel-group orbit through the highest weight line. 
In particular, $V^r_{\omega_d}$ is isomorphic to $V^{\sigma_{r+1}}_{\om_d}$ for an
appropriate choice of the Weyl group element $\sigma_{r+1}$. 

\begin{theoremA}\label{thm:truncrep}
For any $m\ge 2$, the truncated representation  $V^{\sigma_{r+1}}_{m\om_d}$ is isomorphic to $V_\bm$,
where $m_r=1$, $m_d=m-1$, and $m_i=0$ otherwise. 
The truncated flag varieties $\Fl^{\sigma_{r+1}}_{m\om_d}$, $m\ge 2$, are all isomorphic to the 
graph closure of the birational isomorphism $\Gr^r\dashrightarrow \Gr_d(V)$.
\end{theoremA}

This theorem provides a description of the generalized
graph closures of Borovik--Sturmfels--Sverisdóttir \cite{BSS25}. Namely, they considered the case
$r=1$ and the birational isomorphism between $\Gr^1\simeq \bP^{d(n-d)}$ and 
$\Gr_d(V)$. We also note that the second statement of the theorem is a special case 
of a general statement which holds for arbitrary truncated flag varieties.

Our paper is organized as follows. In Section \ref{sec:prelim} we collect the preliminary material on Lie algebras, representations, and flag varieties. 
In Section \ref{sec:BSS} we describe a special case of our approach and recover
certain results of \cite{Fe24}. Section \ref{sec:polytopes} is devoted to the study of
the polytopes $X_\bm$. In Section \ref{sec:rep} we introduce modules $V_\bm$ and construct
their monomial bases. Section \ref{sec:var} is devoted to the study of truncated 
Grassmannians and flag varieties.

\section*{Acknowledgments}
EF was partially supported by the ISF grant 493/24. IM was partially supported by the SFB-TRR 195 ``Symbolic Tools in Mathematics and their Application''.

\subsection*{AI use disclosure} The technique of applying the RSK correspondence and Greene's theorem to integer points of FFLV polytopes and their generalizations was suggested by ChatGPT 5.6 Sol. This technique originally appeared within a purported proof of a statement that was eventually deemed false. However, it was then successfully applied (by the authors) in a number of other arguments that appear in this paper. AI was also used for copyediting purposes.

\section{Preliminaries}\label{sec:prelim}
We work over the field of 
complex numbers. For a vector $w$ in a vector space $W$, we denote by $[w]\in \bP(W)$ the line
spanned by $w$.
We fix two positive integers $d,n$ such that $d\le n-d$. 
For two  integers $k_1,k_2$, we denote by $[k_1,k_2]$ the set of 
integers $\{k_1,\dots,k_2\}$. We also use the standard notation $[k]=[1,k]$.
For a partition $\la$, we denote by $\ell(\la)$ the number of its nonzero parts 
and by $|\la|$ the sum of all parts.

\subsection{Lie algebras}
Let $\fg=\msl_n$ and let $\fn_-\subset \fg$ be the subalgebra of strictly lower triangular
matrices. 
We denote by $\fb=\fn\oplus\fh$ and $\fb_-=\fn_-\oplus\fh$ the upper and lower Borel
subalgebras.
Let $\Phi=\Phi_+\sqcup \Phi_-$ be the set of roots written as a disjoint
union of positive and negative roots. If $\al_1,\dots,\al_{n-1}$ are simple roots,
then the elements of $\Phi_+$ are given by $\al_{i,j}=\al_i+\dots +\al_j$, 
$1\le i\le j<n$. The corresponding Chevalley generators are denoted by 
$e_{\al_{i,j}}=e_{i,j}\in\fn$ and $f_{\al_{i,j}}=f_{i,j}\in \fn_-$. 

Let $V$ be the standard (vector) representation of $\msl_n$. 
We fix a basis 
$v_1,\dots,v_n$ of the space $V$ and consider the decomposition $V= V_d\oplus V_{n-d}$,
\[
V_d = \spa\{v_1,\dots,v_d\},\ V_{n-d}=\spa\{v_{d+1},\dots,v_n\} 
\]
(recall that we always assume $1\le d\le n-d$).

Let $\fa_d\subset\fn_-$ be the
abelian radical spanned by the matrix units $E_{i,j}$ with $1\le j\le d<i\le n$.
In other words, $\fa_d$ is spanned by Chevalley generators $f_{i,j}$ with 
$i\le d\le j$; one has a natural identification
\[
\fa_d\simeq V_d^* \T V_{n-d} \simeq \Hom(V_d,V_{n-d}).
\]
We denote by $\fp_d\subset\fg$ the complementary parabolic subalgebra, i.e., $\msl_n = \fa_d\oplus\fp_d$ (as vector spaces). 
The opposite parabolic subalgebra containing 
$\fa_d$ is denoted by $\fp^-_{d}$. Clearly, the intersection 
$\fp_d\cap \fp^-_{d}$ is equal to the Levi subalgebra $\msl_d\oplus\msl_{n-d}$. 

Sometimes it will be more convenient to work with $\mgl_n$ (as opposed to
$\msl_n$). Slightly abusing notation, we denote by $\fp^-_d$ the corresponding
parabolic subalgebra of $\mgl_n$ as well. In particular, $\fp^-_d$ contains 
the direct sum $\mgl_d\oplus \mgl_{n-d}=\mgl(V_d)\oplus\mgl(V_{n-d})$.

\subsection{Highest weight representations}
Let $\om_1,\dots,\om_{n-1}$ be the fundamental weights of $\msl_n$. Let $P_+$ be 
the set of dominant integer weights $\la$, i.e., $\la=\sum_{i=1}^{n-1} m_i\om_i$,
$m_i\in\bZ_{\ge 0}$. For $\la\in P_+$, we denote by $V_\la$ the corresponding 
irreducible finite-dimensional $\msl_n$-module with highest weight vector $v_\la$.
For any two weights $\la,\mu\in P_+$, the tensor product $V_\la\T V_\mu$ contains
a submodule isomorphic to $V_{\la+\mu}$ as the Cartan component generated by the tensor product
of highest weight vectors $v_\la\T v_\mu$ (as an $\msl_n$-module or as an $\fn_-$-module).

We also work with modules of general linear algebras. Since we consider multiple such algebras, we denote an irreducible representation of $\mgl(V)$ corresponding to
a partition $\la=(\la_1\ge \dots \ge \la_n)$ by $\Sigma_\la V$. 
Clearly, the restriction of $\Sigma_\la V$ to $\msl(V)$ is equal to 
$V_{\sum_{i=1}^{n-1} (\la_i-\la_{i+1})\om_i}$.
Similarly, we 
use the notation $\Sigma_\mu V_d$ and $\Sigma_\nu V_{n-d}$ (for $\mu$ having 
at most $d$ parts and $\nu$ having at most $n-d$ parts). 

Since our central geometric objects are Grassmannians, irreducible representations
whose highest weight is a multiple of a fundamental weight will play a distinguished 
role in our paper. For $m\in\bZ_{>0}$, consider  $V_{m\om_d}$ and let 
$v_{m\om_d}\in V_{m\om_d}$ be the highest weight vector.
For $m=1$, the module $V_{\om_d}$
is realized as the wedge power $\bigwedge^d V$. We denote by $v_I$, 
$I=(i_1,\dots,i_d)$, $1\le i_1<\dots<i_d\le n$, the wedge products 
$v_{i_1}\wedge\dots\wedge v_{i_d}$. Clearly, the vectors $v_I$ form a basis of $V_{\om_d}$. 
The highest weight vector of $V_{\om_d}$ is $v_{\om_d}=v_1\wedge \dots \wedge v_d$.

For any $m>0$, the representation $V_{m\om_d}$ admits a $\fg$-equivariant embedding into 
the $m$-th tensor power of $V_{\om_d}$:
\[
V_{m\om_d}\subset V_{\om_d}^{\T m},\ v_{m\om_d} \mapsto v_{\om_d}^{\T m}. 
\]
One easily sees that $V_{m\om_d}=\U(\fa_d)v_{m\om_d}$. Since $\fa_d$ is abelian,
the universal enveloping algebra coincides with the symmetric algebra $S(\fa_d)$ and hence
$V_{m\om_d}$ inherits the degree grading from the symmetric algebra:
$V_{m\om_d}=\bigoplus_{s\ge 0} V_{m\om_d}(s)$. The space $V_{m\om_d}(0)$ is spanned by
the highest weight vector $v_{m\om_d}$ and the action of an element $E_{i,j}\in\fa_d$
increases the degree by one.

Let $\la=(m,\dots,m,0,\dots,0)=(m^d0^{n-d})$ be the $\gl(V)$ weight corresponding to $m\om_d$ 
($m$ appears $d$ times). The restriction of $\Sigma_\la V$ to $\mgl(V_d)\oplus \mgl(V_{n-d})$
is given by the following Cauchy-type formula 
\begin{equation}\label{eq:Cauchy}
\Sigma_{(m^d0^{n-d})}(V) \simeq \Sigma_{(m^d)} V_d \T \sum_{\substack{\mu=(\mu_1,\dots,\mu_d)\\ \mu_1\le m}} 
\Sigma_\mu V_d^*\T \Sigma_\mu V_{n-d},
\end{equation}
see \cite[Proposition 2.3.1]{W03} (unlike this reference, our right-hand side dualizes one of the spaces and, therefore, uses the same Schur functor in both factors and also has a one-dimensional prefactor).

\subsection{Opposite Demazure modules}
Let $W=S_n$ be the type $A$ Weyl group. For $\sigma\in W$, let $v_{\sigma\la}\in V_\la$ be an extremal
weight vector of weight $\sigma\la$ (this weight space is known to be one-dimensional). 
The (opposite) Demazure module $D_{\sigma\la}\subset V_\la$ is the 
cyclic $\fb_-$-module generated by $v_{\sigma\la}$. Since $v_{\sigma\la}$ is a weight vector, 
the module $D_{\sigma\la}$ is cyclic as an $\fn_-$-module as well. We note that for two
weights $\la,\mu\in P_+$, one has the embedding of $\fb_-$-modules 
$D_{\sigma(\la+\mu)}\subset D_{\sigma\la}\T D_{\sigma\mu}$, where the cyclic vector
$v_{\sigma(\la+\mu)}$ maps to the tensor product $v_{\sigma(\la)}\T v_{\sigma(\mu)}$.
This embedding of Demazure modules is compatible with the Cartan embedding 
$V_{\la+\mu}\subset V_\la\T V_\mu$.

\begin{rem}
The standard Demazure modules are defined as cyclic $\fb$-modules (as opposed
to $\fb_-$-modules) generated 
by extremal weight vectors. 	
We use opposite Demazure modules since we would like them to be cyclic with respect
to the algebra of generating operators $\fb_-$ (i.e., the same algebra as for $V_\la$).
\end{rem}

In this paper we consider special Demazure modules in $V_{m\om_d}$, which are invariant
with respect to the Levi subalgebra $\mgl(V_d)\oplus \mgl(V_{n-d})$.
One easily sees that a Demazure module $D_{\sigma\la}$, $\la=m\om_d$, is determined by
the set $\sigma(\{1,\dots,d\})$. The $\mgl(V_d)\oplus \mgl(V_{n-d})$-invariant 
Demazure modules are parametrized by $r=0,\dots,d$ and correspond
to $\sigma\in S_n$ such that $\sigma(\{1,\dots,d\})=I^r$, where
\[
I^r = (1,\dots,d-r,d+1,\dots,d+r).
\]   
For later use, we fix a representative $\sigma_r\in W$ such that 
$\sigma_r(\{1,\dots,d\})=I^r$ and denote the Demazure module $D_{\sigma\la}$,
$\sigma=\sigma_r$, $\la=m\omega_d$ by 
$D^r_m$. In particular, $D^0_m=V_{m\om_d}$; we also formally set $D^{d+1}_m = 0$.
 
For $r=0,\dots,d$, let us define the truncated representation
\[
V^r_{m\om_d} = V_{m\om_d}/D_m^{r+1}.
\]
By definition, $V^r_{m\om_d}$ is a cyclic $\fb_-$-module generated by a cyclic vector that is 
(the class of) the highest weight vector $v_{m\om_d}$.
To stress the dependence on $r$, we denote the cyclic vector of $V_{m\om_d}^r$
by $v_{m\om_d}^r$. 
 
\begin{example}\label{ex:extreme}
The extreme cases of the truncated representations $V^r_{m\om_d}$ are $r=0$ with
$V^0_{m\om_d}=\spa\{v_{m\om_d}^0\}$ and $r=d$ with
$V^d_{m\om_d}=V_{m\om_d}$.
\end{example} 

For $I\in \binom{[n]}{d}$, we define the degree of $I$ as the cardinality of the set
$I_{>d}=I\cap [d+1,n]$. We also say that the degree of $v_I\in \bigwedge^d V$ is equal
to the degree of $I$.

\begin{rem}
Let $m=1$. Then $V_{\om_d}=\bigwedge^d V$ and $D_1^{r+1}$ is the span of vectors $v_I$ 
of degree at least $r+1$ (i.e., $|I_{>d}|> r$). In turn, the space  $V^r_{\om_d}$ 
is the span of (the classes of) vectors $v_I$  of degree at most $r$, 
i.e., $|I_{>d}|\le r$.	
\end{rem}

One has the following simple lemma.
\begin{lem}\label{lem:thealgebra}
The truncated representations $V^r_{m\om_d}$ are cyclic $\fa_d$-modules.
The action of $\fa_d$ extends to an action of the parabolic subalgebra 
$\fp_d^-$.    	
\end{lem}
\begin{proof}
The first claim is implied by the cyclicity of $V_{m\om_d}$ with respect to 
$\fa_d$. The second claim follows from the invariance of the Demazure modules $D_m^{r+1}$ 
under the action of $\mgl(V_d)\oplus \mgl(V_{n-d})$.   
\end{proof}

\subsection{Truncated Grassmannians}
Let $G=GL_n=GL(V)$ and let $N\subset B$ and $N_-\subset B_-$ be the upper and lower triangular
unipotent and Borel subgroups. The Lie group $\exp(\fa_d)$ is an abelian unipotent
subgroup of $N_-$, $\exp(\fa_d)\simeq \bG_a^{d(n-d)}$, where $\bG_a$ is the 
additive group of the base field. We denote by $P^-_d\supset \exp(\fa_d)$ the maximal parabolic 
subgroup corresponding to $\omega_d$, whose Lie algebra is $\fp^-_d\supset \fa_d$.
In particular, $P^-_d$ contains the group $GL_d\times GL_{n-d}=GL(V_d)\times GL(V_{n-d})$.
 
The truncated Grassmannian $\Gr^r_d(V)$ is defined as the closure of the orbit
of the cyclic line  in the projectivization of $V^r_{\om_d}$:
\[
\Gr^r = \Gr^r_d(V) = \overline{\exp(\fa_d) [v_{\om_d}^r]} \subset \bP(V^r_{\om_d}).
\]
Thanks to Lemma \ref{lem:thealgebra}, the truncated Grassmannians $\Gr^r$ are acted upon by the parabolic
subgroup $P_d^-=(GL_d\times GL_{n-d})\ltimes \exp(\fa_d)$.

\begin{example}
For $r=0$, $\Gr^r_d(V)$ is just a point. For $r=1$, $\Gr^r_d(V)$ is isomorphic to
the projective space $\bP^{d(n-d)}$.  For $r=d$, $\Gr^r_d(V)$ is isomorphic to
the Grassmannian $\Gr_d(V)$. For all $r > 0$, the truncated Grassmannians are pairwise birationally isomorphic (they all share the same open dense orbit of $\exp(\fa_d)$,
since the stabilizer of $[v_{\om_d}^r]$ is trivial for $r>0$).
\end{example} 

Truncated Grassmannians can be explicitly defined in the following way.
For $r=1,\dots,d$, let 
\[
\Lambda^r = \bigoplus_{i=0}^r \Hom\left(\bigwedge\nolimits^i V_d,\bigwedge\nolimits^i V_{n-d}\right).
\]
Let ${\mathbbm 1}$ be a basis vector in $\Hom(\bigwedge^0 V_d,\bigwedge^0 V_{n-d})$.
We have natural maps
\begin{equation}\label{eq:fr}
f_r: \Hom(V_d,V_{n-d}) \to \bP(\Lambda^r),\ 
A \mapsto \left[{\mathbbm 1} + A + \frac{1}{2!}\bigwedge\nolimits^2 A + \dots + \frac{1}{r!}\bigwedge\nolimits^r A\right].
\end{equation}

\begin{lem}\label{lem:likecol}
$\Gr^r$ is isomorphic to the closure of the image of $f_r$.    
\end{lem}
\begin{proof}
One has  a  natural isomorphism $\Lambda^r\simeq V^r_{\om_d}$. In fact,
the span of the degree $i$ vectors $v_I$ is naturally isomorphic to  
$\Hom(\bigwedge^i V_d,\bigwedge^i V_{n-d})$. Now, as we have mentioned above, 
$\fa_d\simeq \Hom(V_d,V_{n-d})$ and under this isomorphism,
$\exp(A)[v^r_{\om_d}]$ corresponds to the right-hand side of \eqref{eq:fr}.
\end{proof}

We also define partial truncated Grassmannians as follows. For a collection 
$\br=(r_1 <\dots < r_t)$ of integers between 1 and $d$, we consider 
\[
\Gr^\br = \overline{\exp(\fa_d)\prod_{1\le i\le t} [v_{\om_d}^{r_i}]} \subset 
\prod_{1\le i\le t} \Gr^{r_i}_d(V).
\]
In particular, if $t=1$, then $\Gr^\br = \Gr^{r_1} = \Gr^{r_1}_d(V)$.
We also note that in the complete case, $\br=[d]$; the variety $\Gr^{[d]}$ is smooth \cite{Fe26}.
It is shown in \cite{Fe26} that partial truncated Grassmannians can be constructed
as blow-ups of the classical Grassmannian $\Gr_d(V)$ along certain Schubert varieties 
(see Section \ref{subsec:general} for more details).

\subsection{Collineations}\label{sec:collin}
Lemma \ref{lem:likecol} shows that truncated Grassmannians can be constructed similarly to varieties of collineations \cite{Th99,Ma20}.
Here we recall the main definitions and formulate basic properties we will need later.

As above, we fix two vector spaces $V_d$ and $V_{n-d}$, $\dim V_d=d$, $\dim V_{n-d}=n-d$ with
$d\le n-d$. For a collection $\br=(r_1 <\dots < r_t)$ of integers between $1$ and $d$,
we define the collineation variety $\mathrm{Col}^\br$ as the image closure of the rational map
\begin{equation}\label{eq:Col}
\Hom(V_d,V_{n-d})\dashrightarrow \prod_{1\le i\le t} \bP\Hom\left(\bigwedge\nolimits^{r_i}V_d,\bigwedge\nolimits^{r_i}V_{n-d}\right),\ 
A\mapsto \prod_{1\le i\le t} \left[\bigwedge\nolimits^{r_i} A\right] 
\end{equation}
(the map is defined on the subset of maximal rank maps).

\begin{example}
The case $t=d$, $\br=[d]$, corresponds to the space of complete collineations 
\cite{LAT89,Ma20}; the complete collineation variety is known to be smooth. 
The other extreme case $t=1$ is studied in \cite{BC09,BCV13}; this is the projectivized
variety of exterior powers of linear maps. Clearly, the complete collineation variety
projects onto all the partial collineation varieties.
\end{example}

We focus on the variety of complete collineations 
$\mathrm{Col}^{[d]}$. By definition, it is embedded into the product of projective varieties
$\bP\Hom(\bigwedge^i V_d,\bigwedge^i V_{n-d})$ for $i\in [d]$.
Let us describe the multihomogeneous coordinate ring $C$ of this embedding. By definition,
$C=\bigoplus_{\bm\in\bZ_{\ge 0}^d} C_\bm$, where $C_\bm$ is the homogeneous component
of all elements of multidegree $\bm$. 

The homogeneous coordinates in $\bP\Hom(\bigwedge^i V_d,\bigwedge^i V_{n-d})$ are labeled by pairs $I,J$ with $I\in {[d+1,n] \choose i}$, $J\in {[d] \choose i}$. The respective coordinate of the image of $A$ under~\eqref{eq:Col} (when defined) is the minor of $A$ spanned by rows $I$ and columns $J$. 
Hence, the ring $C$ can be presented as the 
subalgebra of the polynomial ring 
\begin{equation}\label{eq:polyring}
\bC[\Hom(V_d,V_{n-d})][t_1,\dots,t_d]    
\end{equation}
generated by elements $t_i \Delta_{I,J}$, where $I,J$ are as above and 
$\Delta_{I,J}$ is the respective minor viewed as a polynomial function on $\Hom(V_d,V_{n-d})$. The component $C_\bm$ is spanned by those products of the generators that contain
each $t_i$ exactly $m_i$ times. 
One easily sees that $C_\bm$ is invariant with respect to the action of 
$\mgl(V_d)\oplus \mgl(V_{n-d})$. 

For two partitions $\la$ and $\Lambda$, we say that $\la$ is dominated by $\Lambda$ 
($\la\le \Lambda$) if $\la_1+\dots +\la_i\le \Lambda_1+\dots +\Lambda_i$ for all $i$.
For a collection $\bm=(m_1,\dots,m_d)$, we define $\Lambda_\bm$ as the partition having $m_i$ 
columns  of height $i$ for all $i$. Explicitly,
\[
\Lambda_\bm = (m_1+\dots +m_d, m_2+\dots +m_d,\dots,m_d).
\]

\begin{prop}\label{prop:colhomr}
One has an isomorphism of $\msl(V_d)\oplus \mgl(V_{n-d})$-modules 
\begin{equation}\label{eq:Cbm}
C_\bm \simeq 
\bigoplus_{\substack{\ell(\la)\le d\\ \la\le \Lambda_\bm\\ |\la|=|\Lambda_\bm|}} 
\Sigma_\la V_d \T \Sigma_\la V_{n-d}^*.
\end{equation}
\end{prop}
\begin{proof}
For $i\in [d]$, let $\mathcal{I}_i$ denote the determinantal ideal in~\eqref{eq:polyring} generated by all $t_i\Delta_{I,J}$ with $|I|=|J|=i$. The degree $\bm$ component $C_\bm$ coincides with the degree $\bm$ component of the
product of powers of ideals $\mathcal{I}_1^{m_1}\dots \mathcal{I}_d^{m_d}$. 
Now we apply 
\cite{DCEP80} (see also \cite{BV88}) where it is shown that the lowest degree component
of $\mathcal{I}_1^{m_1}\dots \mathcal{I}_d^{m_d}$ is isomorphic to the right hand side of~\eqref{eq:Cbm}. 
\end{proof}

\begin{rem}
We note that the algebra generated by elements $\Delta_{I,J}$ with $|I|=|J|=r$ for a fixed number $r$ is considered in \cite{BC01}.    
\end{rem}

\subsection{The RSK correspondence}\label{sec:RSK}

We recall several basic properties of the RSK correspondence, referring the reader to~\cite[Chapter 4]{Fu97} for a precise definition and further details.

Slightly specializing the classical construction, we view the RSK correspondence as a bijection between the following two sets:
\begin{enumerate}
\item the set of $(n-d)\times d$ matrices over $\bZ_{\ge0}$ with rows labeled by $[d+1,n]$ and columns labeled by $[d]$,
\item the set of (ordered) pairs $(T,U)$ of semistandard Young tableaux of the same shape, with $T$ on the alphabet $[d]$ and $U$ on the alphabet $[d+1,n]$.
\end{enumerate}
The RSK correspondence maps the former set to the latter; the inverse map is the \textit{inverse RSK correspondence}. 

One key property of this bijection is that, for a matrix $\bs$ in (i) and the corresponding pair $(T,U)$ in (ii), the content of $T$ and $U$ is determined by $\bs$ as follows. The number of entries in $T$ that equal $j\in[d]$ is the total of column $j$ of $\bs$, while the number of entries in $U$ that equal $i\in[d+1,n]$ is the total of row $i$ of $\bs$. In particular, the common size of $T$ and $U$ is equal to the sum of all elements of $\bs$. 

With a matrix $\bs$ in (i) one associates the \textit{two-line array} $\left(\begin{smallmatrix}i_1 & \dots & i_m \\ j_1 & \dots & j_m \end{smallmatrix}\right)$ that contains the column $\begin{smallmatrix}i\\j\end{smallmatrix}$ exactly $s_{i,j}$ times and is lexicographically ordered (i.e., columns $\begin{smallmatrix}i\\j\end{smallmatrix}$ are ordered first by $i$ increasing, then by $j$ increasing). Composing with the above bijection (i)$\to$(ii), we obtain another standard form of the RSK correspondence which replaces (i) with
\begin{itemize}
\item[(i\textquotesingle)] the set of lexicographically ordered two-line arrays with first-line entries in $[d+1,n]$ and second-line entries in $[d]$.
\end{itemize}
Whether \textit{RSK correspondence} refers to the bijection (i)$\to$(ii) or the bijection (i\textquotesingle)$\to$(ii) should be clear from the context. Note that for the second bijection the two resulting tableaux are filled with, respectively, the second-line entries and the first-line entries of the two-line array.

The following theorem characterizes the common shape of the two tableaux corresponding to a two-line array.
\begin{thm}[Greene's theorem~\cite{G74}]\label{thm:Greene}
For a two-line array $\left(\begin{smallmatrix}i_1 & \dots & i_m \\ j_1 & \dots & j_m \end{smallmatrix}\right)$ in (i\textquotesingle), let the corresponding tableaux $T$ and $U$ have shape $\la$. Then, for $k\in\bZ_{\ge0}$:
\begin{enumerate}[label=(\alph*)]
\item the total length of the first $k$ rows of $\lambda$ is equal to the maximal possible size of a disjoint union of $k$ nondecreasing subsequences of $(j_1,\dots,j_m)$;
\item the total length of the first $k$ columns of $\lambda$ is equal to the maximal possible size of a disjoint union of $k$ (strictly) decreasing subsequences of $(j_1,\dots,j_m)$.
\end{enumerate} 
\end{thm}

Next, let $P_d$ be the poset consisting of pairs $(i,j)$ with $i\in [d+1,n]$, $j\in [d]$, equipped with product order on $[d+1,n]\times[d]$. $P_d$ is in natural bijection with the set of roots
of $\fa_d$ and also with the entries of any matrix in (i). 

Recall that, by Dilworth's theorem \cite{Dil50}, the width of a poset is equal to the smallest number of chains into which it may be partitioned. 
% Also recall that, by Mirsky's theorem \cite{Mi71}, the height of a poset is equal to the smallest number of antichains into which it may be partitioned. 
Combining this result with Theorem~\ref{thm:Greene} and the natural bijection (i)$\to$(i\textquotesingle), we deduce
\begin{cor}\label{cor:Greene}
For a matrix $\bs$ in (i), let the corresponding tableaux $T$ and $U$ have shape $\la$. Then, for $k\in\bZ_{\ge0}$:
\begin{enumerate}[label=(\alph*)]
\item the total length of the first $k$ rows of $\lambda$ is equal to \[\max_{\substack{P\subset P_d\\ w(P)\le k}}\ \sum_{(i,j)\in P} s_{i,j};\]
\item the total length of the first $k$ columns of $\lambda$ is equal to 
\[\max_{\substack{\text{\rm antichains }A_1,\dots,A_k\subset P_d\\\mathbbm 1_{A_1}+\dots+\mathbbm 1_{A_k}\le\bs}}\ \sum_{i=1}^k |A_i|,\]
where $\mathbbm 1$ denotes the indicator vector and the inequality below the maximum holds coordinatewise.
\end{enumerate} 
\end{cor}

\section{The boundary case}\label{sec:BSS}
The goal of this section is to present an alternative approach to the construction
in \cite{BSS25} (see also \cite{Fe24}). This approach will be generalized in the following 
sections.

For $m,M\ge 0$, let $V_{m,M}=V_{\om_d}^{\odot m} \odot (V^1_{\om_d})^{\odot M}$
(clearly, $V_{\om_d}^{\odot m}\simeq V_{m\om_d}$ and 
$V^1_{\om_d}\simeq \spa\{v^1_{\om_d}\}\oplus \fa_d$). 
We denote by $v_{m,M}\in V_{m,M}$ 
the cyclic vector of $V_{m,M}$ (the product of highest weight vectors).

We first define the polytopes that are shown to provide monomial bases of the $V_{m,M}$ in \cite{Fe24}.
Recall the poset $P_d$ introduced in Subsection~\ref{sec:RSK}. We consider the polytopes $X_{m,M}\subset \bR_{\ge 0}^{P_d}$
defined by the following inequalities: for any subposet $P\subset P_d$ one has
\[
\sum_{\al\in P} x_\al \le m w(P)+ M.
\]
We define $S_{m,M}=X_{m,M}\cap \bZ^{P_d}$. In particular, $X_{m,0}=X_m$ is the
FFLV polytope for weight $m\om_d$ \cite{FFL11}; the defining inequalities in the case 
$M=0$ reduce to those labeled by width-one subposets, i.e., chains.

The following statement is proved in \cite{Fe24} in a rather elaborate way: 
the elements $f^\bs v_{m,M}$,
$\bs\in S_{m,M}$ form a basis of $V_{m,M}$. One of the crucial steps of the proof
is the following equality: the set $S_{m,M}$, $M\ge 1$ decomposes into the Minkowski
sum $S_{m,M-1}+S_{0,1}$ (note that $S_{0,1}$ consists of integer points of the 
unit simplex). Here we give a conceptual proof of this equality.

\begin{lem}\label{lem:RSK}
The set $S_{m,M}$ is in bijection with the set of pairs of semistandard Young tableaux
$(T,U)$ satisfying the following properties:
\begin{itemize}
\item $T$ and $U$ are of the same shape $\la=(\la_1\ge\dots\ge \la_d\ge 0)$,
\item for any $w=1,\dots,d$, $\sum_{i=1}^w \la_i\le wm+M$,
\item the entries of $T$ are in $[d]$, the entries of $U$ are in $[d+1,n]$.
\end{itemize}
\end{lem}
\begin{proof}
The claim follows via applying the RSK correspondence to elements of $S_{m,M}$ viewed as elements of (i), see Subsection~\ref{sec:RSK}. The second property is a consequence of Corollary~\ref{cor:Greene}(a).
\end{proof}

\begin{prop}\label{prop:M-1}
For $\bs\in S_{m,M}\setminus S_{m,M-1}$, let $A_1,\dots,A_m$ be a collection of antichains in $P_d$ with $\mathbbm 1_{A_1}+\dots+\mathbbm 1_{A_m}\le\bs$ for which $\sum |A_i|$ is maximal (cf.\ Corollary~\ref{cor:Greene}(b)). 
There exists $\gamma\in P_d$ such that $\mathbbm 1_{A_1}+\dots+\mathbbm 1_{A_m}< s_\gamma$. 
For any such $\gamma$, $\bs - {\mathbbm 1}_\gamma\in S_{m,M-1}$.
\end{prop}
\begin{proof}
Let $(T,U)$ be the pair of semistandard tableaux corresponding to $\bs$ under the RSK 
correspondence. Let $\la$ be the common shape of $T$ and $U$. Denote by $H$ the sum of the first $m$ columns of $\la$,
i.e.,
\[
H = \la^t_1+\dots +\la^t_m,\quad\text{where } \la^t \text{ denotes the transposed diagram}.
\]
We also have $H=\sum |A_i|$.

Suppose the contrary to our first claim, i.e., that $H=\sum s_{i,j}$. Since $|\la|=\sum s_{i,j}$, this implies $\la^t_{m+1}=0$. Hence, $\la_1\le m$ and by Corollary~\ref{cor:Greene}(a), the sum of entries of $\bs$ over any chain is at most $m$. We derive that $\bs\in S_{m,0}$, which contradicts the assumption $\bs\notin S_{m,M-1}$.

% We note that $\sum_{j>m} \la^t_j\le M$, i.e., that the sum of lengths of all columns of $\la$ starting from the $(m+1)$-st is at most $M$. 
Next, we apply Corollary~\ref{cor:Greene}(a) to express the sum of all $\la_i>m$:
\[
\sum_{i=1}^{\la^t_{m+1}} \la_i = \max_{w(P)=\la^t_{m+1}} \sum_{\al\in P} s_\al \le
m \la^t_{m+1} + M,
\]
where the last inequality holds since $\bs\in S_{m,M}$.
Now, 
\[
\sum_{i=1}^{\la^t_{m+1}} \la_i = m \la^t_{m+1} + \sum_{j>m} \la^t_j,
 \]
which implies that $\sum_{j>m} \la^t_j\le M$.

Let $\mu$ be the shape of the semistandard tableaux corresponding to the matrix
$\bs - {\mathbbm 1}_\gamma$ via RSK, where  $(\mathbbm 1_{A_1}+\dots+\mathbbm 1_{A_m})_\gamma< s_\gamma$. Corollary~\ref{cor:Greene}(b) implies that $\mu^t_1+\dots+\mu^t_m=H$ (the same collection $A_1,\dots,A_m$ still achieves the maximum).
Since $|\mu|=|\la|-1$, we conclude that 
\[
\sum_{j>m} \mu^t_j = \sum_{j>m} \la^t_j -1 \le M-1.
\]
Hence, for any $w\ge 1$,
\[
\mu_1+\dots +\mu_w \le  mw + \sum_{j>m} \mu^t_j \le mw + M -1.
\]
The claim $\bs-\mathbbm 1_\gamma\in S_{m,M-1}$ now follows via yet another application of Corollary~\ref{cor:Greene}(a).
\end{proof}

\begin{cor}\label{cor:minksum}
The sets $S_{m,M}$ satisfy the Minkowski sum property $S_{m_1,M_1}+S_{m_2,M_2} = 
S_{m_1+m_2,M_1+M_2}$.
\end{cor}
\begin{proof}
Proposition \ref{prop:M-1} implies that $S_{m,M}=S_{m,M-1}+S_{0,1}$. Now we iterate the last equality and then apply the property 
$S_{m_1,0}+S_{m_2,0} = 
S_{m_1+m_2,0}$ of FFLV polytopes \cite{FFL11}.
\end{proof}

Next, recall that the modules $V_{m,M}$ admit a natural $\mgl(V_d)\oplus\mgl(V_{n-d})$-action. 
Explicit decompositions 
of these modules are computed in the following corollary. We use the notation
\[
\Lambda_{m,M} = (m+M,m,\dots,m)=(m+M,m^{d-1}). 
\]
Recall that for partitions $\la$, $\Lambda$, we write $\la\le \Lambda$ if 
$\la_1+\dots+\la_i\le \Lambda_1+\dots+\Lambda_i$ for all $i$.

In the next corollary and in further sections we consider isomorphisms of $\msl(V_d)\oplus \mgl(V_{n-d})$-modules obtained as restrictions of $\mgl(V_d)\oplus \mgl(V_{n-d})$-modules. Such an isomorphism means that the $\mgl(V_d)\oplus \mgl(V_{n-d})$-modules differ by a one-dimensional tensor factor.

\begin{cor}
One has an isomorphism of $\msl(V_d)\oplus \mgl(V_{n-d})$-modules
\[
V_{m,M}\simeq \bigoplus_{\substack{\ell(\la)\le d\\ \la\le \Lambda_{m,M}}} 
\Sigma_{\la} V_d^*\T \Sigma_\la V_{n-d}.
\]
\end{cor}
\begin{proof}
Since $\la$ has at most $d$ parts and $d\le n-d$, both  
$\Sigma_\la V_d$ and $\Sigma_\la V_{n-d}$ are well defined. 
Hence, it is enough to prove the corresponding character identity.

Let us consider the  $\mgl(V_d)\oplus \mgl(V_{n-d})$-character. The weight of the highest weight
vector of $V_{m,M}$ is equal to $((m+M)^d)$ with respect to $\mgl(V_d)$ and to $0$ with respect 
to $\mgl(V_{n-d})$ (since the $\mgl_n$-weight of the highest weight vector of $V_{\om_d}$ is $1^d$).
Now consider $\bs\in S_{m,M}$ and the corresponding pair of tableaux $(T,U)$. 
The coordinates of the $\mgl(V_{n-d})$-weight of $f^\bs v_{m,M}$ are the row sums of $\bs$. Hence, this weight is equal to the content of $U$. A coordinate of the $\mgl(V_d)$-weight of $f^\bs v_{m,M}$ is equal to the difference of $m+M$ and the respective column sum of $\bs$. Hence, the $\mgl(V_d)$-weight of $f^\bs v_{m,M}$ 
is equal to $(m+M)^d$ minus the content of $T$. Summing over all $\bs\in S_{m,M}$ and
using Lemma \ref{lem:RSK}, we see that the characters coincide up to a monomial factor.
\end{proof}

\section{Polytopes}\label{sec:polytopes}
In this section we introduce a family of poset polytopes. In the following
sections we use the combinatorics of these polytopes to study cyclic tensor products
and orbit closures.

\subsection{Chain-type description}
For $\bm=(m_1,\dots,m_d)$, $m_r\in \bZ_{\ge 0}$, we introduce the partition 
\[
\Lambda_\bm = (m_1+\dots +m_d,m_2+\dots +m_d,\dots,m_d).
\]
For a poset $P$ of width $d$, we define the polytope 
$X_\bm(P)\subset\bR_{\ge 0}^{P}$ by the following inequalities
\begin{equation}\label{eq:Xm}
\sum_{\al\in P'} s_\al \le \sum_{i=1}^d \min(i,w(P')) m_i\quad\text{for all } P'\subset P.     
\end{equation}
We note that the right-hand side of this inequality is equal to the total length of the first $w(P')$ rows of $\Lambda_\bm$.  
We set $X_\bm=X_\bm(P_d)$ and denote the integer-point set of $X_\bm$ by $S_\bm$. 

\begin{example}\label{ex:chain}
Let $\bm=(0,\dots,0,m_d)$. Then $X_\bm$ is a scaled chain polytope  for $P_d$ and also an FFLV polytope.
If $\bm=(m_1,0,\dots,0,m_d)$, then
$X_\bm=X_{m,M}$ for $m=m_d$, $M=m_1$.
\end{example}

\begin{lem}\label{lem:RSK-gen}
The set $S_\bm$ is in bijection with the set of pairs $(T,U)$ 
of semistandard Young tableaux satisfying the following properties: 
\begin{itemize}
\item $T$ and $U$ are of the same shape $\la=(\la_1\ge\dots\ge \la_d\ge 0)$,
\item the entries of $T$ are in $[d]$, the entries of $U$ are in $[d+1,n]$,
\item $\la \le \Lambda_\bm$.
\end{itemize}
\end{lem}
\begin{proof}
Given $\bs\in S_\bm$, we apply the RSK correspondence to $\bs$ (viewed as a matrix in the set (i), see Subsection~\ref{sec:RSK}) and apply Corollary~\ref{cor:Greene}(a).     
\end{proof}

\begin{cor}\label{cor:cardS}
The cardinality of $S_\bm$ is equal to the sum over $\la\le \Lambda_\bm$ of the products
$\dim \Sigma_\la V_d^*\cdot \dim \Sigma_\la V_{n-d}$.
\end{cor}

\begin{example}
By Corollary~\ref{cor:minksum}, the sets $S_\bm$ are Minkowski-additive with respect to $\bm$ whenever the only
nonzero entries are $m_1$ and $m_d$. We now give an example showing that an analogous Minkowski sum property does not hold for general $\bm$. 

Let $n=12$, $d=6$,
$\bm=(0,1,0,0,0,1)$. We show that $S_\bm \ne S_{\bm'}+S_{\bm''}$ for
$\bm'=(0,0,0,0,0,1)$, $\bm''=(0,1,0,0,0,0)$.
The following collection $\bs$ belongs to $S_\bm$:
\[
s_{11,1}=s_{10,2}=s_{7,3} = s_{12,4} = s_{9,5} = s_{8,6} = 1, \ s_{i,j}=0 
\text{ otherwise}.
\]
We visualize this collection as follows (the columns of the matrix are labeled by
$1,\dots,6$ and the rows by $7,\dots,12$):
\[
\begin{tikzpicture}
  \draw[step=0.5cm] (0,0) grid (3,3);
  \foreach \p in {
    (0.25,0.75), (0.75,1.25), (1.25,2.75), (1.75,0.25), (2.25,1.75), (2.75,2.25) 
  } 
  {\fill \p circle (2.5pt);  }
\end{tikzpicture}
\]
One easily sees that the sum over any chain is at most two, the sum over 
any width-two set is at most four, and the sum over any width-three set is at most five.
This confirms that $\bs\in S_\bm$. 

The set $S_{\bm'}$ is the set of integer points in the chain polytope (i.e., 
consists of indicator vectors of antichains) and
$S_{\bm''}$ is its subset consisting of points with total sum of coordinates
not exceeding two. One easily sees that the support of $\bs$ cannot be expressed 
as a union of two antichains such that the cardinality of one of them is
at most two.
\end{example}

\subsection{GT-type polytopes}\label{sec:GT}
As above, we fix a collection of nonnegative integers 
$\bm=(m_1,\dots,m_d)\in\bZ_{\ge 0}^d$ and the corresponding
partition $\Lambda_\bm$. 
Let $M=m_1+\dots +m_d$ be  the length of the largest part of $\Lambda_\bm$.
We consider the set $SSYT_{d\times M}$ of rectangular semistandard Young tableaux of height $d$ and width $M$
on the alphabet $[n]$.
In particular, such tableaux parametrize a basis in the dual module $V_{M\om_d}^*$,
the degree $M$ component of the homogeneous coordinate ring of the Grassmannian $\Gr_d(V)$.

Consider $T\in SSYT_{d\times M}$. We partition $T$ into a disjoint union of two skew tableaux
$T^1$ and $T^2$, where $T^1$ consists of cells in $T$ with entries in $[d]$ and $T^2$
consists of cells with entries in $[d+1,n]$. Clearly, $T^1$ is a semistandard tableau.
Let $\mu=(\mu_1\ge\dots\ge \mu_d)$ be the shape of $T^1$.
Our goal now is to construct two semistandard tableaux $\bar T^1$, $\bar T^2$ 
of shape $\nu=(M-\mu_d,M-\mu_{d-1},\dots, M-\mu_1)$; the entries of $\bar T^1$ will be in
$[d]$, the entries of $\bar T^2$ will be in $[d+1,n]$.

Our first step is to transform $T^2$ into $\bar T^2$. We define $\bar T^2$ as a tableau of shape $\nu$ given by
\[
\bar T^2_{i,j} = n+d+1 - T_{d+1-i,M+1-j}.
\]
In words, we first rotate $T^2$ by $180$ degrees and then reflect the entries with respect
to the middle of the segment $[d+1,n]$. One easily sees that $\bar T^2$ is a semistandard 
tableau of shape $\nu$ whose entries are in $[d+1,n]$.

Now we transform $T^1$ into $\bar T^1$, defined as follows. The first (leftmost) column
of $\bar T^1$ consists of elements of $[d]$ that are not present in the $M$-th (rightmost)
column of $T^1$; the elements are placed in increasing order from top to bottom. 
The second column of $\bar T^1$ consists of elements of $[d]$ that are not 
present in the $(M-1)$-st  column of $T^1$ (placed in increasing order), and so on. It is easily checked that $\bar T^1$ is also semistandard.

Thus, we have constructed a map $\Psi$ from $SSYT_{d\times M}$ to the set of those pairs of semistandard tableaux in (ii) (see Subsection~\ref{sec:RSK}) whose common shape $\la$ satisfies $\la_1\le M$. Inverting the construction of $\bar T_1$ and $\bar T_2$, one sees that $\Psi$ is a bijection.
We next consider the preimage of those pairs of tableaux whose 
shape is dominated by $\La_\bm$. 

\begin{lem}\label{lem:Psidominated}
For $T\in SSYT_{d\times M}$, the common shape of the tableaux 
in $\Psi(T)$ is 
dominated by $\La_\bm$ if and only if the following conditions hold for all $a\in[1,d-1]$:
\begin{equation}\label{eq:SSYTbm}
\#\{(i,j):\ i\ge d-a,\ T_{i,j}\le d \} \ge am_1 + (a-1)m_2 + \dots + m_a.
\end{equation}
\end{lem}
\begin{proof}
Recall the decomposition of $T$ into two tableaux $T^1$ and $T^2$. The shape of $T^1$
is denoted by $\mu$. We are interested in the condition 
\[
\nu = (M-\mu_d,M-\mu_{d-1},\dots,M-\mu_1)\le \Lambda_\bm = (m_1+\dots + m_d,\dots, m_d).
\]
Since $M=\sum_{r=1}^d m_r$, these inequalities are equivalent to 
\[
\mu_d + \dots + \mu_{d-a} \ge am_1 + (a-1)m_2 + \dots + m_a,\ a\in [d-1].
\]
To complete the proof, we note that, by definition, $\mu_i$ is the number of entries in the $i$-th row of $T$ that do not exceed $d$.  
\end{proof}

Now let us pass from semistandard tableaux to Gelfand--Tsetlin (GT) patterns. We use the standard  
bijection: given a tableau $T\in SSYT_{d\times M}$, we attach to it a collection 
of numbers $\kappa^s_i$ defined by the formula
\begin{equation}\label{eq:kappaT}
\kappa^s_i = \#\{j\in [M]:\ T_{i,j}\le s\},\ s\in [n], i\in [s].
\end{equation}
Clearly, one has $\kappa^{s+1}_i\ge \kappa^s_i\ge \kappa^{s+1}_{i+1}$, which, by definition, means that the values $\kappa^s_i$ form a GT pattern. 
We denote the set of all GT patterns obtained in this way by $GT_M$ and the bijection from $SSYT_{d\times M}$ to $GT_M$ by $\Phi$. 
% We make the following straightforward observations.

If $i>d$, then $\kappa^s_i=0$. If $i\in [d]$, $s\ge n-d+i$, then  $\kappa^s_i=M$.
Hence, the non-fixed entries are $\kappa^s_i$ with $i\in [d]$, $i\le s < n-d+i$, of which there are $d(n-d)$. In particular, $(\kappa^n_1,\dots,\kappa^n_n)=(M^d0^{n-d})$ and $GT_M$ is precisely the set of GT patterns corresponding to the $\mgl(V)$-weight $(M^d0^{n-d})$.

Consider the poset $\tilde P_d$ consisting of pairs $(i,s)$, $i\in [d]$, $i\le s < n-d+i$ with covering
relations $(i,s+1)\ge (i,s)\ge (i+1,s+1)$.
The set $GT_M$ is naturally the integer-point set of the $M$-fold scaling of the order polytope of $\tilde P_d$. We denote this scaled order polytope by $Y_M$; this is also the GT polytope of the weight $(M^d0^{n-d})$. Note that $\tilde P_d\simeq P_d$, which is yet another way of seeing that $|GT_M|=|S_{(0,\dots,0,M)}|$, the latter consisting of integer points in the scaled chain polytope of $P_d$ (see Example \ref{ex:chain}).

\begin{prop}
For $T\in SSYT_{d\times M}$, the common shape of the tableaux in $\Psi(T)$ is dominated by $\Lambda_\bm$ if
and only if the entries $\kappa^s_i$ of $\Phi(T)$ satisfy the following inequalities
for all $a\in [d-1]$:
\begin{equation}\label{eq:GTineq}
\kappa^d_d + \kappa^d_{d-1} + \dots +\kappa^d_{d-a} \ge am_1 + (a-1) m_2 + \dots + m_a. 
\end{equation}
\end{prop}
\begin{proof}
Recall Lemma~\ref{lem:Psidominated}. By  \eqref{eq:kappaT}, the left-hand side
of \eqref{eq:SSYTbm} is equal to the left-hand side of \eqref{eq:GTineq}. This
gives the desired claim.
\end{proof}

\begin{rem}
A standard way to visualize a GT pattern is as a triangle with $\kappa^n_1,\dots,\kappa^n_n$ forming the first horizontal row, $\kappa^{n-1}_1,\dots,\kappa^{n-1}_{n-1}$ forming the second horizontal row, and so on. Thus, the above inequalities only affect entries in a single row of the GT pattern $\Phi(T)$.
\end{rem}

We denote by $Y_\bm$ the subpolytope of $Y_M$ cut out by the inequalities~\eqref{eq:GTineq} and by $GT_\bm\subset GT_M$ its integer-point set. To summarize, we have obtained the following chain of bijections between the integer-point sets of $X_\bm$ (a subpolytope of the scaled chain polytope or FFLV polytope) and $Y_\bm$ (a subpolytope of the scaled order polytope or GT polytope):
\begin{gather*}
S_\bm, \text{ the integer-point set of } X_\bm\subset \bR^{P_d}\ \eqref{eq:Xm}\\
\downarrow \text{RSK}\\
\text{ pairs of semistandard tableaux on alphabets } [d], [d+1,n] \text{ of shape } 
\la\le \La_\bm\\
\downarrow \Psi^{-1}\\
\text{ semistandard tableaux of shape } d\times M \text{ satisfying } \eqref{eq:SSYTbm}\\
\downarrow \Phi\quad\\
GT_\bm, \text{ the integer-point set of } Y_\bm\subset\bR^{\tilde P_d}.
\end{gather*}

\subsection{Examples and remarks}
\begin{example}\label{ex:n=4d=2}
Let $n=4$, $d=2$, $\bm=(m_1,m_2)$, $M=m_1+m_2$. Then the polytopes $X_\bm$ and $Y_\bm$ are defined by the inequalities:
\begin{gather*}
(s_{12},s_{22},s_{23},s_{13})\in\bR_{\ge 0}^4:\\
s_{12}+s_{22}+s_{23}\le M,\ s_{12}+s_{13}+s_{23}\le M,\\ s_{12}+s_{22}+s_{23} + s_{13}\le m_1+2m_2
\end{gather*}    
and
\begin{gather*}
(\kappa^1_1,\kappa^2_1,\kappa^2_2,\kappa^3_2)\in\bR_{\ge 0}^4:\\
M\ge \kappa^2_1\ge \kappa^1_1\ge \kappa^2_2,\ \kappa^2_1\ge \kappa^3_2\ge \kappa^2_2,\
\kappa^2_1 + \kappa^2_2 \ge m_1.
\end{gather*} 
The constructed bijection between $S_\bm$ and $GT_\bm$ can be written as the following piecewise linear map:
\begin{gather*}
s_{12} = \kappa^2_1 - \max(\kappa^1_1,\kappa^3_2),\quad
s_{23} = \min(\kappa^1_1,\kappa^3_2)-\kappa^2_2,\\
s_{22} = M - \kappa^2_1 + \max(\kappa^3_2-\kappa^1_1,0),\quad
s_{13} = M - \kappa^2_1 + \max(\kappa^1_1-\kappa^3_2,0).
\end{gather*}
The inverse bijection is given by
\begin{gather*}
\kappa^2_1 = M - \min(s_{22},s_{13}),\quad 
\kappa^1_1 = \kappa^2_1 - s_{12} - \max(s_{13}-s_{22},0),\\
\kappa^3_2 = \kappa^2_1 - s_{12} - \max(s_{22}-s_{13},0),\quad
\kappa^2_2 = \kappa^2_1 - s_{12} - |s_{22}-s_{13}| - s_{23}.
\end{gather*}
In fact, these formulas define a piecewise linear bijection between the polytopes $X_\bm$ and $Y_\bm$ and one may similarly find such bijections for general~$\bm$.
\end{example}

\begin{rem}
Let us describe the composite correspondence between $S_\bm$ and
$GT_\bm$ when $M=1$, i.e., $\bm$ has a single nonzero entry $m_r=1$.
In this case, $\bs\in S_\bm$ is the indicator vector of an antichain $(i_1,j_1),\dots,(i_s,j_s)$ with
$s\le r$, $i_1<\dots <i_s$, $j_1>\dots >j_s$. The two corresponding tableaux under RSK are the columns
$(j_s,\dots,j_1)$ and $(i_1,\dots,i_s)$. The respective $d\times M$ semistandard tableau $T$ is a height-$d$ column obtained as follows. Its upper $d-s$ entries form the
set $[d]\setminus \{j_1,\dots,j_s\}$; its lower $s$ entries are 
$d+n+1-i_s,\dots,d+n+1-i_1$. The entry $\kappa^s_i$ of the corresponding GT pattern $\Phi(T)$ is equal to 1 if the $i$-th entry of the above height-$d$ column 
is at most $s$ and is zero otherwise.
\end{rem}

\begin{rem}
For $\bm=(0,\dots,0,M)$ our composite map provides a bijection between the integer-point sets of the scaled chain and order polytopes of $P_d$. This is, in general, \textbf{not} the standard Stanley bijection~\cite{Sta86}, as seen already in Example \ref{ex:n=4d=2}.
Instead, we obtain a bijection that is similar to the one appearing in, e.g.,~\cite[\S 3]{NY04}.
\end{rem}

\section{Representations}\label{sec:rep}
We define the $\fp^-_d$-module $V_\bm$ as follows:
\[
V_\bm = \bigodot_{r=1}^d (V^r_{\om_d})^{\odot m_r}.
\]
We denote the cyclic vector of $V_\bm$ by $v_\bm$.
\begin{example}
If $\bm=(0,\dots,0,m_d)$, then $V_\bm\simeq V_{m_d\om_d}$ (as usual, we assume $d\le n-d$).
If $\bm=(m_1,0,\dots,0)$, then $V_\bm\simeq \bigoplus_{i=0}^{m_1} S^i(\fa_d)$.
If $\bm=(m_1,0,\dots,0,m_d)$, then $V_\bm\simeq V_{m,M}$ with $m=m_d$, $M=m_1$.
\end{example}

We would like to prove that the vectors $f^\bs v_\bm$, $\bs\in S_\bm$ form a basis of $V_\bm$. 
Here $f^\bs = \prod_{\al\in P_d} f_\al^{s_\al}=\prod_{i,j} E_{i,j}^{s_{i,j}}$,
where $(i,j)\in P_d$, i.e., $i\in [d+1,n]$, $j\in [d]$.
We first show that they span the whole space.

\begin{lem}\label{lem:Nw}
For $w\in[d]$, let $N_w=\sum_{i=1}^d \min(i,w) m_i$. Then for any
$\beta_1,\dots,\beta_w\in P_d$ one has 
\[
f_{\beta_1}^{b_1}\dots f_{\beta_w}^{b_w}v_\bm =0\quad \text{ if }\quad b_1+\dots+b_w > N_w. 
\]
\end{lem}
\begin{proof}
By definition, 
\[
N_w = m_1+2m_2+\dots + (w-1)m_{w-1} + w(m_w+\dots +m_d).
\]
Recall the cyclic vectors $v^r_{\om_d}\in V^r_{\om_d}$, $r\in [d]$. The cyclic vector $v_\bm$ 
is explicitly written as 
$v_\bm=\bigotimes_{r=1}^d (v^r_{\om_d})^{\T m_r}$. Hence, to compute
$f_{\beta_1}^{b_1}\dots f_{\beta_w}^{b_w}v_\bm$ one distributes the factors $f_{\beta_i}$ among the vectors $v^r_{\om_d}$ appearing in $v_\bm$ in all possible ways. 

First, note that applying more than $m_1+2m_2+\dots + (w-1)m_{w-1}$ factors to
$\bigotimes_{r=1}^{w-1} (v^r_{\om_d})^{\otimes m_r}$ produces zero, since the largest degree of a vector 
in $V^r_{\om_d}$ is equal to $r$. Hence, it suffices to prove that 
\[
f_{\beta_1}^{c_1}\dots f_{\beta_w}^{c_w} \bigotimes_{r=w}^d (v^r_{\om_d})^{\otimes m_r} = 0
\]
for $c_1+\dots +c_w > w (m_w+\dots + m_d)$. 
This inequality implies that there exists an $i\in [w]$ such that $c_i>m_w+\dots + m_d$.
Since $f^2_\beta v^r_{\om_d}=0$ for all $r$ and $\beta$, we conclude that 
$f_{\beta_i}^{c_i} \bigotimes_{r=w}^d (v^r_{\om_d})^{\otimes m_r} = 0$ for $c_i>m_w+\dots + m_d$.
\end{proof}

\begin{cor}\label{cor:span}
Vectors $f^\bs v_\bm$, $\bs\in S_\bm$, span $V_\bm$.    
\end{cor}
\begin{proof}
The proof largely repeats that of \cite[Proposition 2.18]{Fe24}. We outline the argument and refer to~\cite{Fe24} for further details. 

Consider $\bs\in\bZ_{\ge 0}^{P_d}$ not contained in $S_\bm$. We find $P'\subset P_d$ for which the respective inequality~\eqref{eq:Xm} is violated. We decompose $P'$ into a disjoint union of $w(P')$ chains. For each chain $C_i$, we consider its supremum (or join) $\beta_i\in P_d$. Set $b_i=\sum_{\alpha\in C_i} s_\alpha$. By Lemma~\ref{lem:Nw}, we have $\prod f_{\beta_i}^{b_i} v_\bm = 0$.

Next, we have an action of $\fb$ on $\fn_-$, viewing the latter as the quotient of $\fb$-modules $\fg/\fb$. One may define a product $K=E_{i_1,j_1}\dots E_{i_N,j_N}\in \U(\fb)$ for which 
\begin{equation}\label{eq:Kprodf}
K\prod f_{\beta_i}^{b_i} = f^\bs + D    
\end{equation} 
holds in $\U(\fb\ltimes\fn_-)$, where $D$ is a sum of PBW monomials in $\U(\fn_-)$ that precede $f^\bs$ under a certain monomial order. Since the left-hand side of~\eqref{eq:Kprodf} annihilates $v_\bm$ and so does any root vector in $\fn_-$ outside $\fa_d$, we conclude by induction on $\bs$ over the mentioned monomial order that $f^\bs v_\bm$ is a linear combination of $f^{\bs'}v_\bm$ with $\bs'\in S_\bm$.
\end{proof}

It remains to prove linear independence of vectors $f^\bs v_\bm$.
To be precise, we show that $\dim V_\bm\ge |S_\bm|$. Together with Corollary 
\ref{cor:span} this would give linear independence.

Recall (see Corollary \ref{cor:cardS}) that the cardinality of $S_\bm$ is equal
to the sum over partitions $\lambda$ dominated by $\Lambda_\bm$ of $\dim R_\la$,
where 
\[
R_\la = \Sigma_\la V_d^*\T \Sigma_\la V_{n-d}.
\]

Let $M=\sum_{i=1}^d m_i$.
\begin{lem}
There exists a surjective homomorphism of $\fp^-_d$-modules $V_{M\om_d}\to V_\bm$
sending the highest weight vector $v_{M\om_d}$ to the cyclic vector $v_\bm\in V_\bm$.
\end{lem}
\begin{proof}
By definition, $V_\bm=\bigodot_{r=1}^d (V_{\om_d}^r)^{\odot m_r}$ and 
$V_{M\om_d}=\bigodot_{r=1}^d (V_{\om_d})^{\odot m_r}$. Each $V^r_{\om_d}$ is a quotient
of $V_{\om_d}$. Hence, $V_\bm$, as  the cyclic product, is also a quotient of the
cyclic product of $(V_{\om_d})^{\odot m_r}$.
\end{proof}

By~\eqref{eq:Cauchy}, $V_{M\om_d}$ admits the following decomposition as an 
$\msl(V_d)\oplus\mgl(V_{n-d})$-module:
\begin{equation}\label{eq:Cauchyrect}
V_{M\om_d} \simeq \bigoplus_{\la\le (M^d)} R_\la,\ (M^d)= (\underbrace{M,\dots,M}_d).
\end{equation}
Hence, as an $\msl(V_d)\oplus\mgl(V_{n-d})$-module, $V_\bm$ decomposes into a direct sum of certain
summands in \eqref{eq:Cauchyrect}. We want to show that the summands appearing in 
$V_\bm$ correspond precisely to $\la$ dominated by $\Lambda_\bm$ (note that $\La_\bm\le (M^d))$.

In what follows, it will be convenient to work with the natural surjection
$\U(\fa_d)\to V_\bm$. Since $\fa_d$ is abelian, the universal enveloping algebra is
isomorphic to a polynomial algebra in the variables $E_{i,j}$. Since 
$\fa_d\simeq V_d^*\T V_{n-d}$, one gets the Cauchy decomposition
\[
\U(\fa_d)\simeq \bigoplus_{\la=(\la_1\ge \dots \ge\la_d\ge 0)} 
\Sigma_\la V_d^*\T \Sigma_\la V_{n-d}.
\]

We will prove that for any $\la\le \Lambda_\bm$ there exists a
vector in $V_\bm$ that projects nontrivially to $R_\la\subset \U(\fa_d)$.
To this end, we recall the following statement from \cite{DCEP80}.

For $k\in[d]$, let $\Delta_k\in \U(\fa_d)$ be the determinant of the $k\times k$ matrix with entries
$(E_{d+i,j})_{i,j=1}^k$; in particular, $\Delta_1=E_{d+1,1}$. The following is proved in~\cite[Theorem 3.3 and Corollary 3.4]{DCEP80}: 

\begin{prop}\label{prop:RlaHWV}
For any partition $\la=(\la_1,\dots,\la_d)$, the product
$\Delta_\la = \prod_{k=1}^d \Delta_k^{\la_k-\la_{k+1}}$ is the highest weight vector of the submodule $R_\la\subset \U(\fa_d)$.
\end{prop}

Our goal is to find a nonzero vector in $V_\bm$ such that its preimage in $\U(\fa_d)$
has a nontrivial projection onto $R_\la$. We start with the following combinatorial lemma.

For a partition $\Lambda$ and $a\le\Lambda_1$, let $\Lambda^{a}$ be the partition defined as follows. We consider the transposed diagram $\Lambda'$, decrease the first $a$ parts of $\Lambda'$ by 1 obtaining a new tuple $\tilde\Lambda'$, and denote the partition obtained by reordering $\tilde\Lambda'$ by $(\Lambda^a)'$. The partition $\Lambda^a$ is defined as the transpose of $(\Lambda^a)'$. In other words, $\Lambda^a$ is obtained from $\Lambda$ by shortening $a$ of its columns by 1 with these columns chosen to be as long as possible.

\begin{lem}\label{lem:cutcell}
Let $\la=(\la_1\ge \dots \ge \la_s>0)$ and $\Lambda=(\Lambda_1\ge \dots \ge \Lambda_r>0)$ 
be two partitions such that $\la\le \Lambda$. Then \[(\la_1,\dots,\la_{s-1})\le\Lambda^{\la_s}.\]
\end{lem}
\begin{proof}
Denote the partial sums of $\Lambda$ and $\la$ by $S_k=\Lambda_1+\dots+\Lambda_k$ and $s_k=\la_1+\dots+\la_k$. Similarly, denote the partial sums of $\Lambda^{\la_s}$ by $S^{\la_s}_k$ and those of $(\la_1,\dots,\la_{s-1})$ by $\ov s_k$. We are to show that $\ov s_k\le S^{\la_s}_k$ for all $k$.

Consider the largest $t$ such that $\Lambda_t\ge\la_s$. We have $S^{\la_s}_k=S_k$ for $k<t$, while $S^{\la_s}_k=S_{k+1} - \la_s$ for $k\ge t$.

Now, if $k< t$, then
\[\ov s_k\le s_k\le S_k = S^{\la_s}_k.\]
If $k\ge s$, then
\[\ov s_k = s_k-\la_s\le S_k-\la_s \le \min(S_k,S_{k+1}-\la_s)\le S^{\la_s}_k.\]
Finally, if $t\le k < s$, then
\[\ov s_k = s_k = s_{k+1}-\la_{k+1}\le s_{k+1}-\la_s\le S_{k+1}-\la_s = S^{\la_s}_k.\qedhere\]
\end{proof}

\begin{prop}\label{prop:nonvanish}
The vector $\Delta_\la v_\bm$ is nonzero if $\la\le \Lambda_\bm$ and $\ell(\la)\le d$.    
\end{prop}
\begin{proof}
Choose $\la\le\Lambda_\bm$. We assume that $\la$ is nonzero, otherwise the claim is trivial. Let $\la_s$ be the last nonzero entry of $\la$.
Consider the decomposition $\Delta_\la = \Delta_\la^\circ + \overline{\Delta_\la}$,
where $\Delta_\la^\circ$ is the sum of those terms in $\Delta_\la$ (as a polynomial
in $E_{i,j}$) that are divisible by $E_{d+s,s}^{\la_s}$.
We note that 
\[
\Delta_\la^\circ = E_{d+s,s}^{\la_s} \Delta_1^{\la_1-\la_2}\dots \Delta_{s-1}^{\la_{s-1}}.
\]
In words, the quotient of $\Delta_\la^\circ$ by $E_{d+s,s}^{\la_s}$ is equal to 
$\Delta_{\bar\la}$, where 
$\bar\la$ is obtained from $\la$ by removing the last (shortest) row. 

For $r\in[d]$, $k\in[m_r]$, let $V^{r,k}$ be a copy of $V^r_{\om_d}$ so that
\[
V_\bm = \bigodot_{r=1}^d \left(V^{r,1}\odot \dots\odot V^{r,m_r}\right)\subset 
\bigotimes_{r=1}^d \left(V^{r,1}\T \dots\T V^{r,m_r}\right).
\]
Thus, any vector in $V_\bm$ is a linear combination of tensor products of the form 
$\bigotimes_{r\in [d]}\bigotimes_{k\in [m_r]} w_{r,k}$ for some $w_{r,k}\in V^{r,k}$, each of the form $v_I$.
Our goal is to find vectors $w^\circ_{r,k}$ of this form for which the coefficient of $\bigotimes_{r\in [d]}\bigotimes_{k\in [m_r]} w^\circ_{r,k}$ in $\Delta_\la v_\bm$ does not vanish.

We denote the highest weight vector of $V^{r,k}$ by $v^{r,k}$; under the natural isomorphism with $V_{\om_d}^r$ it corresponds to (the class of) 
$v_1\wedge\dots\wedge v_d$. We define $w^\circ_{r,k}$ as follows. Let $r_0$ be  the largest $r$ for which
$m_r>0$. Let us order all pairs $(r,k)$, $r\in [d]$, $k\in [m_r]$ first by $r$ 
in the decreasing order and then by $k$ in decreasing order, i.e.,
\[
(r_0,m_{r_0}), (r_0,m_{r_0}-1), \dots, (r_0,1), (r_0-1,m_{r_0-1}), (r_0-1,m_{r_0-1}-1),\dots. 
\]
We choose the first $\la_s$ pairs $(r,k)$ from this list and for these set 
\[
w^\circ_{r,k} = E_{d+s,s} v^{r,k} = v_{1,\dots,s-1,d+s,s+1,\dots,d}.
\]
For all remaining pairs $(r,k)$, we set $w^\circ_{r,k}=v^{r,k}$. 
We denote the tensor product $\bigotimes_{r,k} w^\circ_{r,k}$ by $w^\circ$.  

One sees that $w^\circ$ is contained in $E_{s+d,s}^{\la_s}v_\bm$ with coefficient $\la_s!$. Indeed, this coefficient is equal to the number of ways to distribute
the $\la_s$ factors $E_{s+d,s}$ among the $\la_s$ chosen pairs $(r,k)$. 
    
Recall the decomposition $\Delta_\la = E_{d+s,s}^{\la_s} \Delta_{\bar\la} + \overline{\Delta_\la}$.
Let us show that
\begin{equation}\label{eq:nonzeroequiv}
\Delta_{\bar\la}w^\circ \ne 0\quad \text{ implies }\quad \Delta_\la v_\bm\ne 0.     
\end{equation}
Indeed, $\Delta_{\bar\la}w^\circ \ne 0$ means that there exists a vector 
$w=\bigotimes_{r,k} w_{r,k}$ with each $w_{r,k}$ a pure wedge product 
$v_{I(r,k)}$, $I(r,k)\in \binom{[n]}{d}$, such that the coefficient of $w$ in $\Delta_{\bar\la}w^\circ$ is nonzero. We claim that the coefficient of $w$ in
$\overline{\Delta_\la}v_\bm$ vanishes. Observe that the collection
$\{I(r,k)\}_{r,k}$ satisfies the following property: there exist exactly $\la_s$ pairs
$(r,k)$ such that $s\notin I(r,k)$, $d+s\in I(r,k)$. Indeed, this property holds for 
$w^\circ$ and remains in place after applying $\Delta_{\bar\la}$.
However, $\overline{\Delta_\la}v_\bm$ contains no summands with this property.
Indeed, $\Delta_\la v_\bm= \prod_{i=1}^s \Delta_i^{\la_i-\la_{i+1}}v_\bm$.
Factors $\Delta_i$, $i<s$ do not involve generators of the forms $E_{a,s}$ and $E_{d+s,b}$, hence, 
as operators on $V_{\om_d}$, they vanish on $v_s$ and never produce $v_{d+s}$.
Furthermore, each monomial in $\Delta_s$ contains exactly one factor of the form $E_{a,s}$ and one factor of the form $E_{d+s,b}$. Thus, the only monomials in $\Delta_\la$ that map $v_\bm$ to a vector with the mentioned property are those divisible by $E_{d+s,s}^{\la_s}$. 

Now it remains to show that $\Delta_{\bar\la}w^\circ \ne 0$. We do so by induction on $|\La_\bm|$, the base $|\La_\bm|=1$ being trivial.
Let $\ov\bm\in\bZ_{\ge 0}^d$ be the collection whose $r$-th entry is the number of height-$r$ 
columns of the diagram $\Lambda_\bm^{\la_s}$ (see Lemma  \ref{lem:cutcell}). 
In other words, $\Lambda_{\ov\bm} = \Lambda_\bm^{\la_s}$.
We claim that $\Delta_{\bar\la}w^\circ \ne 0$ in $V_\bm$ if and only if 
$\Delta_{\bar\la}v_{\ov \bm} \ne 0$ in $V_{\ov\bm}$. 
In fact, we show that the $\fa'$-modules generated by $v_{\ov\bm}$ and $w^\circ$ are isomorphic, where $\fa'$ is spanned by $E_{i,j}$ with $d+1\le i < d+s$, $1\le j <s$. This would imply the desired claim, since all the matrix units $E_{i,j}$ appearing in $\Delta_{\bar\la}$ belong to $\fa'$.
 
Recall that $w^\circ$ differs from $v_\bm$ by the substitution
\[
v^{r,k} \mapsto v_1\wedge\dots \wedge v_{s-1}\wedge v_{d+s}\wedge v_{s+1}\wedge \dots \wedge v_d
\]
in $\la_s$ positions. Let us denote the right-hand side by $u^{r,k}$.
The map  $u^{r,k}\mapsto v^{r-1}_{\om_d}$ extends to an $\fa'$-module isomorphism $\U(\fa')u^{r,k}\simeq \U(\fa')v^{r-1}_{\om_d}$.
Indeed, all matrix units $E_{i,j}$ appearing in $\fa'$ satisfy $d+1\le i < d+s$, $1\le j <s$ 
and hence vanish on $v_s$ and never map to $v_{d+s}$. Since $u^{r,k}$ is obtained from $v_{\om_d}^r$ by replacing a single $v_i$, $i\in[d]$ with $v_j$, $j>d$, it is annihilated by any PBW monomial in $\U(\fa')$ of degree greater than $r-1$.
This provides the isomorphism $\U(\fa')u^{r,k}\simeq \U(\fa')v^{r-1}_{\om_d}$.

Now, recalling the definition of $\ov\bm$, note that $v_{\ov\bm}$ is obtained from $w^\circ$ by replacing each factor of the form $u^{r,k}$ with the respective $v^{r-1}_{\om_d}$. Hence, 
we obtain the desired isomorphism $\U(\fa')w^\circ\simeq \U(\fa')v_{\ov \bm}$.

Finally, by Lemma \ref{lem:cutcell},  
$\bar\la\le \Lambda_{\ov\bm}$. Thus, the induction hypothesis combined with~\eqref{eq:nonzeroequiv} provides $\Delta_{\bar\la}v_{\ov\bm} \ne 0$. 
\end{proof}

\begin{cor}\label{cor:showsup}
Let $\la$ be a partition dominated by $\Lambda_\bm$. Then $R_\la$ is a direct 
summand of the $\msl(V_d)\oplus\mgl(V_{n-d})$-module $V_\bm$.
\end{cor}
\begin{proof}
Let $V_\bm = \U(\fa_d)/I_\bm$, i.e., $I_\bm$ is the ideal of polynomials vanishing on 
$v_\bm$. Then $I_\bm$ is $\mgl(V_d)\oplus\mgl(V_{n-d})$-invariant, and hence it is equal
to a direct sum of certain spaces $R_\mu$. Propositions~\ref{prop:RlaHWV} and \ref{prop:nonvanish} imply that
the summands $R_\la$, $\la\le \Lambda_\bm$ cannot occur in the decomposition of $I_\bm$ and hence occur in that of $V_\bm$.
\end{proof}

Recall the set $S_\bm$ of integer points of the polytope $X_\bm$ defined by \eqref{eq:Xm}.

\begin{thm}\label{thm:main}
Vectors $f^\bs v_\bm$, $\bs\in S_\bm$ form a basis of $V_\bm$.
Furthermore, one has an isomorphism of $\msl(V_d)\oplus\mgl(V_{n-d})$-modules
\[
V_\bm\simeq \bigoplus_{\substack{\la\le \Lambda_\bm\\ \ell(\la)\le d}} 
\Sigma_\la V_d^*\T \Sigma_\la V_{n-d}.
\]
\end{thm}
\begin{proof}
By Corollary \ref{cor:span}, vectors $f^\bs v_\bm$, $\bs\in S_\bm$ span $V_\bm$.
Hence, in view of Corollary \ref{cor:cardS}, 
\[
\dim V_\bm\le \sum_{\substack{\la\le \Lambda_\bm\\ \ell(\la)\le d}} 
\dim \Sigma_\la V_d^* \cdot \dim\Sigma_\la V_{n-d}.
\]
Corollary \ref{cor:showsup} gives the opposite inequality and hence implies equality.
We conclude that the vectors $f^\bs v_\bm$, $\bs\in S_\bm$ form a basis of $V_\bm$ as a spanning set of the appropriate cardinality.
\end{proof}

\section{Varieties}\label{sec:var}
In this section we study the geometric properties of truncated Grassmannians 
and flag varieties. We start with a subsection treating the general case (not necessarily
type $A$ Grassmannians) and then return to the more specific setting.

\subsection{Schubert varieties}\label{subsec:general}
In this section, $\fg$ is an arbitrary simple Lie algebra with simply connected Lie group $G$.
Let $B\supset H\subset B_-\subset G$ be the Borel subgroup, the Cartan torus, and 
the opposite Borel subgroup. The corresponding Lie subalgebras of $\fg=\mathrm{Lie}(G)$  
are denoted by $\fb\supset \fh\subset \fb_-$.
For a parabolic subgroup $P$ containing $B$,
we denote by $G/P$ the respective (partial) flag variety of $G$.
For a Weyl group element $\sigma\in W$, $G/P$ contains an $H$-fixed point
$p_\sigma$ (these points may coincide for different $\sigma$ when $P\ne B$). 
Let $X_\sigma\subset G/P$ be the opposite Schubert variety, i.e., $X_\sigma$
is the closure of the orbit $B_-p_\sigma$.

\begin{example}
For $\sigma=\mathrm{id}$, one has $X_{\mathrm{id}}=G/P$. For the longest element 
$w_0\in W$, one has $X_{w_0}=p_{w_0}$.
\end{example}

Let $\la\in\fh^*$ be a dominant integral weight and let $V_\la$ be the corresponding 
irreducible highest weight $\fg$-module. 
For $\sigma\in W$, let $v_{\sigma\la}\in V_\la$ be the extremal weight vector of weight
$\sigma\la$. 
We denote by $D_{\sigma\la}\subset V_\la$
the opposite Demazure module, $D_{\sigma\la}=\U(\fb_-)v_{\sigma\la}$.
The quotient $V_\la/D_{\sigma\la}$ will be denoted by $V^\sigma_\la$. In particular,
$V^\sigma_\la$ is a cyclic $\fb_-$-module with the cyclic vector being (the class of) 
the highest weight vector $v_\la$.

For a dominant integral $P$-regular weight $\la$, 
let us consider the $G$-equiva\-riant embedding $G/P\subset \bP(V_\la)$.
The corresponding homogeneous coordinate ring $A_\la$ is given by 
$\bigoplus_{m\ge 0} V_{m\la}^*$; the multiplication in $A_\la$ is induced 
by the $\fg$-equivariant embedding $V_{(m+s)\la}\subset V_{m\la}\T V_{s\la}$
sending $v_{(m+s)\la}$ to $v_{m\la}\T v_{s\la}$. We denote by $I_\la\subset S^\bullet(V^*_\la)$
the defining ideal of the embedding $G/P\subset \bP(V_\la)$;
in particular, $A_\la\simeq S^\bullet(V_\la^*)/I_\la$.

Each Schubert variety $X_\sigma\subset G/P$ is also embedded into $\bP(V_\la)$; 
in particular, the point $p_\sigma$ corresponds to $[v_{\sigma\la}]$.
By \cite{R87} (see also \cite{BK05}), the image of $X_\sigma$ is the intersection 
of the embedded $G/P$ with the projective subspace $\bP(D_{\sigma\la})$. Moreover, the 
reduced scheme structure of $X_\sigma\subset \bP(V_\la)$ is defined by the ideal 
$I_{\sigma\la}$ generated by $I_\la$ and $\mathrm{Ann}(D_{\sigma\la})\subset V_\la^*$.
For the homogeneous coordinate ring $A_{\sigma\la}=S^\bullet(V_\la^*)/I_{\sigma\la}$ 
of  the Schubert variety $X_\sigma$, one has a decomposition 
$A_{\sigma\la}=\bigoplus_{m\ge 0} D^*_{m\sigma\la}$.

Since $I_{\sigma\la}$ contains $I_\la$, we can consider its image in $A_\la$;
we denote the resulting ideal by the same symbol.
One has  the degree decomposition 
$I_{\sigma\la}=\bigoplus_{m\ge 0} I_{\sigma\la}(m)$, $I_{\sigma\la}(m)\subset A_\la(m)$.
By the above discussion, $I_{\sigma\la}$ is generated inside $A_\la$ by the degree-one part $I_{\sigma\la}(1)$,
which is equal to $\mathrm{Ann}(D_{\sigma\la})$. We obtain the following proposition.

\begin{prop}
For any $\la$, $\sigma$ and $m>1$, one has $V_{m\la}^\sigma\simeq V_{(m-1)\la}^\sigma\odot V_\la$.     
\end{prop}
\begin{proof}
We note that we have homomorphisms of $\fb_-$-modules
\[
I_{\sigma\la} \simeq \bigoplus_{m\ge 0} (V_{m\la}^\sigma)^*\subset \bigoplus_{m\ge 0} (V_{m\la})^*
\simeq A_\la.
\]
Indeed, 
\[
(V_{m\la}^\sigma)^* = (V_{m\la}/D_{m\sigma\la})^*\simeq 
\mathrm{Ann}(D_{m\sigma\la})\subset V_{m\la}^*.
\]
However, $I_{\sigma\la}(m)\subset V_{m\la}^*$ is exactly the annihilator of $D_{m\sigma\la}$.

Since $I_{\sigma\la}$ (as an ideal in $A_\la$) is generated in degree one, we have surjective
$\fb_-$-equivariant morphisms $I_{\sigma\la}(m-1)\T A_\la(1)\to I_{\sigma\la}(m)$ for any $m\ge 2$.
Dualizing, we obtain the desired isomorphism
$V_{m\la}^\sigma\simeq V_{(m-1)\la}^\sigma\odot V_\la$ (because, after dualization, the highest weight vector is mapped to the tensor product of highest weight vectors).
\end{proof}

\begin{rem}
We note that $V_{(m-1)\la}^\sigma\odot V_\la^\sigma$ is \textbf{not}, in general, isomorphic to 
$V_{m\la}^\sigma$. The simplest example is $\fg=\msl_2$, $\la=\omega$, $m=2$,
$\sigma = s_1$. Then $\dim V_\la^\sigma\odot V_\la^\sigma=1$ and 
$\dim V_{2\la}^\sigma =2$.    
\end{rem}

Now, we consider truncated flag varieties 
\[
{F}_\la^\sigma = \overline{B_- [v_\la]} \subset \bP(V_\la^\sigma),
\]
where $[v_\la]$ is the line spanned by $v_\la$. 
We obtain the following corollary. Let $F_\la$ denote $\ov{B_-[v_\la]}\subset\bP(V_\la)$, the partial flag variety corresponding to $\la$. 

\begin{cor}\label{cor:higher}
${F}_{m\la}^\sigma\simeq {F}_{(m+1)\la}^\sigma$ for $m>1$. This variety is 
isomorphic to the closure of the graph of the natural birational isomorphism
${F}_{\la}^\sigma\dashrightarrow F_\la$.
\end{cor}
\begin{proof}
For $m>1$ one has
$V_{m\la}^\sigma\simeq V_{(m-1)\la}^\sigma\odot V_\la$ and
\[
V_{(m+1)\la}^\sigma\simeq V_{(m-1)\la}^\sigma\odot V_\la \odot V_\la \simeq V_{(m-1)\la}^\sigma\odot V_{2\la}.  
\]
We conclude that ${F}_{(m+1)\la}^\sigma$ sits inside the product 
${F}_{(m-1)\la}^\sigma\times \bP(V_{2\la})$ and 
${F}_{m\la}^\sigma$ sits inside the product 
${F}_{(m-1)\la}^\sigma\times \bP(V_{\la})$, and in both cases
the corresponding varieties are realized as closures of $B_-$-orbits
through the highest weight lines. However, one has a $B_-$-equivariant
isomorphism
\[
\overline{B_-[v_{2\la}]}\simeq \overline{B_-[v_{\la}]},\
[v_{2\la}]\mapsto [v_\la].
\]
We also note that $\overline{B_-[v_\la]}=F_\la$ and hence, for any $m>1$,  $F_{m\la}^\sigma$
embeds into the product $F_{\la}^\sigma\times F_\la$ as the $B_-$-orbit closure of $[v^\sigma_\la]\times [v_\la]$.
\end{proof}

\subsection{Higher truncated Grassmannians}
Recall that the truncated Grassmannians $\Gr^r$ are defined inside the 
projectivization of the truncated fundamental representations $V^r_{\om_d} = V_{\om_d}^{\sigma_{r+1}}$. 
The higher truncated Grassmannians are defined as $F_{m\om_d}^{\sigma_{r+1}}$ sitting
inside $\bP(V^r_{m\om_d})=\bP(V^{\sigma_{r+1}}_{m\om_d})$.
By Corollary \ref{cor:higher}, for any $m>1$ the higher truncated Grassmannian
is isomorphic to the closure of the graph of the birational isomorphism 
$\Gr^r\dashrightarrow \Gr_d(V)=\Gr^d$
(to be compared with the $r=1$ case of \cite{BSS25,Fe24}). In other words, 
$F_{m\om_d}^{\sigma_{r+1}}\simeq \Gr^\br$ for the tuple $\br=(r,d)$.

Now consider the corresponding cyclic representations.
From the general setup of subsection \ref{subsec:general} we know that 
$V^r_{m\om_d}\simeq V^r_{(m-1)\om_d}\odot V_{\om_d}$ for $m\ge 2$. Hence,
\begin{equation}\label{eq:Vrm}
V^r_{m\om_d}\simeq V_{(r^1d^{m-1})} = V^r_{\om_d}\odot V_{(m-1)\om_d}.
\end{equation}

Let $\bm$ be the tuple in the middle term of \eqref{eq:Vrm}, i.e., 
$m_r=1$, $m_d=m-1$ and $m_i=0$ otherwise.
The partition $\Lambda_\bm$ in this case is given by 
\[
\Lambda_\bm = (\underbrace{m,\dots,m}_r,m-1,\dots,m-1).
\]
Hence the monomial basis $\prod_{\al}f_\al^{s_\al} v_\bm$ is parametrized by 
the set $S_\bm$ described by the inequalities
\[
\sum_{\al\in P} s_\al\le (m-1)w(P) + \min(r,w(P)).
\]

\begin{rem}
The tableaux $T\in SSYT_{d\times m}$ corresponding to elements of $S_\bm$ (see Subsection~\ref{sec:GT}) are those satisfying
$T_{d-r,1}\le d$. Indeed, the conditions \eqref{eq:SSYTbm} state that the number of pairs
$(i,j)$ such that $i\ge d-r-p,\ T_{i,j}\le d$ is at least $p+1$ for $p=0,\dots,d-r-1$; for $p=0$,
this is equivalent to $T_{d-r,1}\le d$, which implies all other inequalities. Hence, the 
complement in $SSYT_{d\times m}$ is given by the condition $T_{d-r,1} > d$. Equivalently, this complement consists of tableaux all of whose columns are elementwise larger than or equal 
to $\sigma_{r+1}([d])=(1,\dots,d-r-1,d+1,\dots,d+r+1)$. We conclude that the complement enumerates the usual Young tableaux basis of the dual of $D_{m\om_d}^{\sigma_{r+1}}$ (see \cite{LMS86,RS97}).    
\end{rem}

\subsection{Homogeneous coordinate rings}
Recall the definition of truncated Grassmannians 
\begin{gather*}
\Gr^r = \Gr_d^r(V) = \overline{\exp(\fa_d) [v^r_{\om_d}]}\subset \bP(V^r_{\om_d}), \ r\in [d],\\
\Gr^\br = \overline{\exp(\fa_d) \prod_{r\in \br}[v^r]} \subset \prod_{r\in \br} \Gr^r,\ \br\subset [d].
\end{gather*}
The base point of $\Gr^\br$ (the line containing the product of cyclic vectors $v^r_{\om_d}$, 
$r\in \br$) will be denoted by $[v^\br]$. By definition, $\Gr^\br$ admits an embedding into the
product of projective spaces $\prod_{r\in \br} \bP(V^r_{\om_d})$.

\begin{prop}\label{prop:coordring}
The multihomogeneous coordinate ring of $\Gr^\br$ is isomorphic to the direct sum 
$\bigoplus_\bm V_\bm^*$, where the sum ranges over tuples $\bm=(m_1,\dots,m_d)$
with $m_r=0$ unless $r\in \br$.
\end{prop}
\begin{proof}
To simplify notation, we prove the most general case $\br=[d]$; the 
partial case of arbitrary $\br$ is similar. For $\br=[d]$, the coordinate ring 
is claimed to be isomorphic to the direct sum of all dual representations $V_\bm^*$.

The direct sum $\bigoplus_\bm V_\bm^*$ is endowed with the natural algebra structure
$V_\bm^*\T V_{\bm'}^*\to V_{\bm+\bm'}^*$ coming from the natural embedding 
$V_{\bm+\bm'}\to V_\bm\T V_{\bm'}$; the embedding sends $v_{\bm+\bm'}$ to
$v_\bm\T v_{\bm'}$ and is equivariant with respect to $\fp_d^-$. One easily
sees that $\bigoplus_\bm V_\bm^*$ is $\bZ_{\ge 0}^d$-graded and is generated by degree-$e_r$ components, $r\in [d]$, where $e_r\in \bZ_{\ge 0}^d$ is the $r$-th standard basis 
vector. Explicitly, the algebra $\bigoplus_\bm V_\bm^*$ is generated by the spaces
$(V^r_{\om_d})^*=V^*_{(0,\dots,1,\dots,0)}$.

The multihomogeneous coordinate ring $\bC[\Gr^{[d]}]$ is also graded by $\bZ_{\ge 0}^d$. Let
us denote its degree $\bm$ component by $H_\bm$ and denote $H_r=H_{e_r}$. Since the $\exp(\fa_d)$-span of $v^r_{\om_d}$
is equal to $V^r_{\om_d}$, one has $H_r\simeq (V^r_{\om_d})^*$. For any $\bm\in\bZ_{\ge 0}^d$, the 
product map $H_1^{m_1}\cdots H_d^{m_d} \to H_\bm$ is surjective and $\fp_d^-$-invariant.
Hence, $H_\bm^*$ embeds into the tensor product of all $(H_r^*)^{\T m_r}$, i.e., $H_\bm^*$ is a submodule of $V_\bm$. Since this submodule also has the same highest weight as $V_\bm$, it must coincide with $V_\bm$.
\end{proof}

We now give an alternative proof of the second statement of Theorem \ref{thm:main}.

\begin{thm}
One has an isomorphism of $\msl(V_d)\oplus \mgl(V_{n-d})$-modules
\begin{equation}\label{eq:Vbm}
V_\bm \simeq \bigoplus_{\substack{\la\le\Lambda_\bm\\ \ell(\la)\le d}} 
\Sigma_\la V_d^*\T \Sigma_\la V_{n-d}.
\end{equation}
\end{thm}
\begin{proof}
By Proposition \ref{prop:coordring}, $V_\bm$ is dual to the degree $\bm$
component of the multihomogeneous coordinate ring of the complete truncated Grassmannian
$\Gr^{[d]}$. This ring can be realized explicitly by computing the multihomogeneous coordinates of $\exp(A)\prod_{r\in[d]}[v^r]\in\Gr^{[d]}$ for $A\in\fa_d\simeq\Hom(V_d,V_{n-d})$. It is the subring of the polynomial ring 
$\bC[\Hom(V_d,V_{n-d})][t_1,\dots,t_d]$ spanned by products of the form
\begin{equation}\label{eq:prodDelta}
t_1^{m_1}\dots t_d^{m_d} \Delta^{(1)}_1\dots \Delta^{(1)}_{m_1} \dots \Delta^{(d)}_1 \dots \Delta^{(d)}_{m_d},
\end{equation}
where each $\Delta^{(r)}_i$ is the polynomial function on $\Hom(V_d,V_{n-d})$ given by a minor
of size at most $r$. Recall (see Subsection \ref{sec:collin}) that  for complete 
collineations the $\bm$-th homogeneous 
component $C_\bm$ is spanned by those products of the form \eqref{eq:prodDelta} in which each minor $\Delta^{(r)}_i$ has size exactly $r$. We thus arrive at the following conclusion:
\[
V^*_\bm\simeq \sum_{\bm'\,:\, \Lambda_{\bm'}\le \Lambda_\bm} C_{\bm'}\subset \bC[\Hom(V_d,V_{n-d})][t_1,\dots,t_d]. 
\]
Now, by Proposition \ref{prop:colhomr}, one has
\[
C_\bm \simeq \bigoplus_{\substack{\ell(\la)\le d\\ \la\le \Lambda_\bm\\ |\la|=|\Lambda_\bm|}} 
\Sigma_\la V_d\T \Sigma_\la V_{n-d}^*.
\]
Hence, we obtain \eqref{eq:Vbm}. 
\end{proof}

\end{document}